\documentclass[pdflatex]{sn-jnl}
\usepackage{amssymb,amsthm,amsmath,bm}

\usepackage{hyperref}
\usepackage{mathrsfs} 
\usepackage{mathptmx}
\usepackage{amsfonts}
\usepackage{longtable}
\usepackage{caption}
\usepackage{subcaption}
\usepackage{enumerate}
\usepackage{algorithm}
\usepackage{algpseudocode}
\usepackage{mathtools}
\usepackage{array}
\usepackage{float}
\usepackage{cite}

\renewcommand{\epsilon}{\varepsilon}
\renewcommand{\phi}{\varphi}
\newcommand{\spt}{\mathop{\rm spt}}

\newcommand{\id}{\mathrm{id}}
\newcommand{\M}{\mathbb{M}}

\def\lip{\mathop{\rm lip}\nolimits}
\def\Tan{\mathop{\rm Tan}\nolimits}

\renewcommand{\P}{\mathcal{P}}
\renewcommand{\R}{{\mathbb R}}   
\renewcommand{\C}{\mathcal{C}}

\newcommand{\Oloc}{\mathcal{O}_{\rm loc}}
\renewcommand{\d}{\,d}
\renewcommand{\div}{{\rm div}}

\newenvironment{proofof}[1]{\smallskip\noindent{\textbf{Proof~of~#1.}}%
  \hspace{1pt}}{\hspace{-5pt}{\nobreak\quad\nobreak\hfill\nobreak%
    $\square$\vspace{2pt}\par}\smallskip\goodbreak}

\jyear{2023}%

\theoremstyle{thmstyleone}%
\newtheorem{theorem}{Theorem}[section]
\newtheorem{proposition}{Proposition}[section]%
\newtheorem{lemma}{Lemma}[section]%

\theoremstyle{thmstyletwo}%
\newtheorem{example}{Example}%
\newtheorem{remark}{Remark}[section]%

\theoremstyle{thmstylethree}%
\newtheorem{definition}{Definition}[section]%

\begin{document}

\title[Optimal control of nonlocal continuity equations: numerical solution]{Optimal Control of Nonlocal Continuity Equations: Numerical Solution}
  
\author*[1]{\fnm{Roman} \sur{Chertovskih}}\email{roman@fe.up.pt}

\author[2]{\fnm{Nikolay} \sur{ Pogodaev}}\email{pogodaev@math.unipd.it}

\author[1]{\fnm{Maxim} \sur{ Staritsyn}}\email{staritsyn@fe.up.pt}

\affil*[1]{\orgdiv{Research Center for Systems and Technologies (SYSTEC),\\  Faculty of Engineering}, \orgname{University of Porto}, \orgaddress{\street{Rua Dr. Roberto Frias,~s/n}, \city{Porto}, \postcode{4200-465},  \country{Portugal}}}

\affil[2]{\orgdiv{Department of Mathematics ``Tullio Levi-Civita''}, \orgname{University of Padova}, \orgaddress{\street{Via Trieste 63}, \city{Padova}, \postcode{35121}, \country{Italy}}}

\abstract{The paper addresses an optimal ensemble control problem for nonlocal continuity equations on the space of probability measures. We admit the general nonlinear cost functional, and an option to directly control the nonlocal terms of the driving vector field. For this problem, we design a descent method based on Pontryagin's maximum principle (PMP). To this end, we derive a new form of PMP with a decoupled Hamiltonian system. Specifically, we extract the adjoint system of linear nonlocal balance laws on the space of signed measures and prove its well-posedness. 
As an implementation of the designed descent method, we propose an indirect deterministic numeric algorithm with backtracking. We prove the convergence of the algorithm and illustrate its modus operandi by treating a simple case involving a Kuramoto-type model of a population of interacting oscillators. }

\keywords{optimal control, nonlocal continuity equations, Pontryagin’s maximum principle, descent method, indirect algorithms for optimal control}

\pacs[MSC Classification]{49K20, 49J45, 93C20}

\maketitle

\begin{center}
\textbf{To the blessed memory of Professor Fernando Lobo Pereira}
\end{center}

\section{Introduction}\label{sec:intro}
Nonlocal continuity equations on the spaces of probability measures arise as 
macroscopic mathematical models of multi-agent dynamical systems describing the time evolution of large ensembles (beams, crowds, swarms, populations, networks) of structurally identical objects (e.g., elementary particles, people, animals, ``neurons'' of natural or artificial neural networks etc.). The main idea is to treat the many-particle dynamics \emph{as a whole} by focusing on its ``statistical'' behavior assuming that the agents are homotypic and, therefore, indistinguishable.

Passing to the limit in the number of agents, a large set of individuals (described by a system of many similar ODEs) is replaced by their continual probability distribution, named the ``mean field'' (driven by a single transport PDE). This idea, rooted in statistical mechanics~\cite{Dobrushin1979},
has been found useful in different areas of applied mathematics such as mathematical biology \cite{Mogilner1999, CuckerSmale, Carrillo2010, Carrillo2014}, modeling of pedestrian and urban traffic \cite{Cristiani,Piccoli,Colombo2011}, mathematical neuroscience \cite{Laing2018} and even theoretical foundations of artificial intelligence \cite{WeinanJiequn,SChPP-2022,Bonnet2021AMT,Pham1}, just to name a few.

Recent results in the analysis on the space of measures, achieved in  the works of L.~Ambrosio, N. Gigli, J. Lott, F. Otto, F. Santambrogio, G.  Savar\'{e}, C.~Villani, and others, have been proved fruitful for mathematical control theory, largely spurred by the variety of mentioned applications and the needs of control engineering. The starting point was the derivation of a mathematically rigorous ``mean field limit'' of the classical multi-agent optimal control problem \cite{Fornasier2014-me,Fornasier2014} (see also  \cite{fornasier2019,Bonnet-Rossi-2021}). In the consequent few years, the cornerstones of the classical optimal control theory --- such as
Pontryagin's maximum principle (PMP) \cite{averboukh2022pontryagin,Bonnet2019,BonnetRossi2019,Colombo2011,Pogodaev2019,POGODAEV20203585,Pogo-ContEq,BonnetFrankowska2021a}, and the dynamic programming method~\cite{Cavagnari2018,Marigonda2019,Averboukh2018-1,Averboukh2018-2} --- were extended to the area of mean field control.

The mean field PMP, which is at the focus of the present paper, was obtained on different levels of generality by various mathematical strategies. Its particular version was first derived in \cite{Pogo-ContEq} for a specific ``shepard's'' problem over the local continuity equation, and subsequently, for a general linear problem with relaxed controls \cite{Pogodaev2019}; these version of PMP are mainly reconstructed from the differential properties of flows of the driving vector field by standard analytical methods such as Filippov's lemma. A result of the similar spirit for another particular local problem was recently obtained in \cite{Bonnet2021AMT} as a specification of a more general PMP \cite{BonnetRossi2019} by an original technique of generalized Lagrange multipliers on the convex subset of Radon measures with unit mass. Notice that, in the local case, PMP takes the familiar form as it is formulated in terms of a certain \emph{decoupled} optimality system with an explicit backward adjoint equation --- a non-conservative transport PDE.

The first result in this line was obtained in \cite{Bongini2017} for a particular ``bi-level'' optimization problem; a natural strategy was 
to pass to the limit in the usual PMP conditions for conventional control problems obtained by the ``finite-agent'' approximations in the Dobrushin's framework. Similar arguments, based on a finite-dimensional approximation and Ekeland's variational principle, were used in \cite{PogStar2022} to prove an impulsive version of PMP for a nonlocal transport equation with states being measure-valued curves of bounded variation. For general (non-impulsive) nonlocal transport equations, PMP was first proved by an appropriate extension of the classical technique of needle-shaped control variations for problems without \cite{BonnetRossi2019} and with  \cite{Bonnet2019} additional state-constraints. Another approach, relying on an appropriate linearization of the nonlocal dynamics, was further proposed in \cite{BonnetFrankowska2021a}. A different method to derive the necessary optimality conditions for mean-field control problems was suggested in \cite{Burger2021} exploiting an appropriate generalization of Karush–Kuhn–Tucker conditions. Also, in \cite{Carmona2015,WeinanJiequn,Siska}, alternative versions of the mean-field PMP were obtained for stochastic optimal control problems. 

\subsection{Numerical solution: mainstream approaches and their pitfalls}

The use of existing analytical methods is limited to the simplest mean-field control problems, while the transition of these results to the numerical context is fraught with critical technical difficulties. Here, PMP would be a promising footing {if} it were not for a number of significant flaws. The key drawback is due to the mentioned coupling in the Hamiltonian system. 
The state of such a Hamiltonian equation~--- a measure on the cotangent bundle of the state space~--- is \emph{always singular}, even if the solution of the primal continuity equation~ --- a measure on the state space~--- has a density. This makes it impossible to solve the Hamiltonian system by the standard numerical schemes and, consequently, the existing forms of PMP do not provide a descent algorithm.

In the finite-dimensional case, a wide range of various direct and indirect numerical methods are described in numerous works. 
For nonlocal continuity equations, the numerical solution of optimal control problems still remains a burning question, which is principal for the transfer of the mean-field control theory to the practice of control engineering. The mainstream approaches are represented by the following two families:
\begin{enumerate}
\item \emph{Semi-direct (finite-particle) method}: Approximation of the initial distribution by a discrete measure and transformation of a distributed control system to a high-dimensional ODE. The resulting finite-dimensional control problem is solved directly or using special techniques such as, e.g., ``random batch''  methods~\cite{Zuazua2021}.

\item \emph{Direct method}: Total discretization of a nonlocal equation and reduction of a variational problem to mathematical programming.
\end{enumerate}

In practice, both the mentioned approaches typically lead to unsatisfactory results. The first one returns one to a high-dimensional classical optimal control problem followed by the ``curse of dimensionality''; in fact, this approach rejects  the very heart of the mean-field approximation along with all profits of the statistical averaging, while it draws us back to the need of keeping track of all individual representative of a large population. The second approach leads to a complex (high-dimensional, nonlinear and non-convex) mathematical programming problem, which is not always satisfactory solved even by commercial solvers. Here, the main difficulty is the presence of non-local terms depending on the density distribution over the entire spatial grid making the computations much more demanding. This feature also leads to a dramatic loss in the efficiency of parallelization, since integration steps require interprocessor communications of the ``all-to-all'' pattern.

In contrast to the classical setting, the bibliography on \emph{indirect} numeric algorithms for optimal mean-field control is poor. 
There are only few results \cite{Bonnet2021AMT,SChPP-2022,Pogo-Arx,Sinigaglia2021OptimalCO,annunziatoFokkerPlanckControl2013}, all focusing on particular problems, and relying on adequate necessary optimality conditions. The work \cite{Pogo-Arx} deals with the so-called ``shepard's problem'', where one has to steer the population of non-interacting individuals to a given target set; the proposed numeric algorithm is based on a specific form of PMP. On the conceptual level, the algorithm \cite{Bonnet2021AMT} (named in the cited paper a ``shooting method'') is a variant of the classical Krylov-Chernous'ko algorithm --- probably the first indirect algorithm based on PMP in the history of optimal control. The convergence of the algorithm essentially depends on the convexity of the cost functional, and is not guaranteed in general, even for the finite-dimensional case $\mu_t = \delta_{x(t)}$. An alternative algorithm was proposed in \cite{SChPP-2022} for the linear problem of ensemble control employing an exact formula for the increment of the cost functional and feedback control variations.  In \cite{Sinigaglia2021OptimalCO}, a version of the gradient descent method was constructed for a mean-field optimal control problem over a nonlocal Fokker-Planck-Kolmogorov equation modeling interactions in a Kuramoto type model: the first variation of the objective functional and the adjoint equation are obtained by a formal Lagrange method due to the model specifics. Finally, to the best of our knowledge, there are no results of this sort for the general \emph{$\mu$-nonlinear} problem.

\subsection{Goals, contribution, and organization of the paper}

In the present work, we put forth an indirect numerical method for optimal mean-field control.  Namely, we design a  PMP-based indirect deterministic numeric algorithm with backtracking line search for a class of optimal ensemble control problems involving nonlocal continuity equations in the space of probability measures. The method can be viewed as an adequate version of the classical gradient descent method, and demonstrates encouraging results in a series of numeric experiments. To our knowledge, this is the first \emph{indirect descent algorithm} for mean-field control problems, \emph{nonlinear} in measure. 

The derivation of the algorithm is based on a set of new theoretical results, which are of independent interest. First, we derive the linearized form of the original nonlocal transport PDE. In contrast to \cite{BonnetFrankowska2021a}, our arguments apply to nonlocal perturbations of the vector field, and therefore, cover the case, when the control is injected into the nonlocal term of the dynamics. As a byproduct, we compute the first variation of the cost functional within the class of weak variations of the control function. Another contribution is a new, equivalent articulation of PMP, where the Hamiltonian equation on the cotangent bundle of the state space is decoupled into the primal (forward) and dual (backward) parts; the dual systems turns to be a system of nonlocal linear balance laws (continuity equations with sources).

The rest of paper is organized as follows: A statement of the optimal control problem is presented in Sect.~\ref{subsec:problem}. Section~\ref{sec:prelim} collects the necessary notation, and several noteworthy facts from the topology, analysis, and differential calculus over the space of probability measures. In Sect.~\ref{subsec:flow_deriv}, we introduce the concept of a flow of a nonlocal vector field and calculate a ``directional derivative'' of the flow along a nonlocal vector field. Sections~\ref{sec:incr}-\ref{sec:descent} dwell on a simplified version of the stated optimization problem, where the running cost rate is lifted, and the driving vector field is affine in the control variable; this technical simplification is not critical but enables us to shorten the presentation of the main results.  

In Section~\ref{sec:incr}, we exhibit two standard representations of the increment of the cost functional. The first one is formulated in the language of flows of nonlocal vector fields, while the second formula is written down in terms of the mentioned Hamiltonian system.
In Section~\ref{sec:adjoint}, noting that none of these representations are suitable for numerical purposes, we derive the third version of the cost increment, which relies on the notion of adjoint equation. The corresponding numerical algorithm is presented in Section~\ref{sec:descent}. We study the convergence of the algorithm, discuss certain principal aspects of its technical implementation and demonstrate its modus operandi by treating a simple but illustrative case, namely, an aggregation problem for a mean-field Kuramoto-type oscillatory model. Finally, in Section~\ref{sec:gen}, the obtained results 
are extended to the general problem, involving the running cost and the nonlinear dependence on the control variable.

\subsection{Problem statement}
\label{subsec:problem}

Given the data $ V\colon I\times \R^{n}\times \P_{2}(\R^{n})\times U\to \R^{n}$,
  $L\colon I\times \R^{n}\times \P_{2}(\R^{n})\times U\to \R$,
  $\ell\colon \P_{2}(\R^{n})\to \R,$
consider the following optimal control problem $(P)$ on a fixed finite time interval $I\doteq [0,T]$:
\begin{gather}
{
\text{Minimize}\quad \mathcal{I}[u]\doteq \int_{0}^{T}L\left(t,\mu_{t},u(t)\right)\d t + \ell(\mu_T)\quad\text{subject to}} \label{eq:cost}\\
\partial_t\mu_t + \div_x \left(V_{t}\left(x,\mu_{t},u(t)\right) \, \mu_t\right) = 0,\quad \mu_0=\vartheta,\label{eq:conteq}\\
u \in \mathcal U.
\label{eq:control}
\end{gather}
We assume that control signals are functions $t \mapsto u(t)$ of time variable only, and take values in a given set $U \subseteq \R^m$, i.e., $\mathcal U \doteq L^\infty(I; U)$, where $L^\infty$ is equipped with the weak* topology $\sigma(L^\infty, L^1)$.

Optimization problems of this sort appear in the framework of multi-agent dynamical systems, where the measure $\mu_t$ represents the spatial distribution of agents at time $t$. The specified class of controls implies that $u$ acts \emph{simultaneously} on all agents (one can imagine that we are able to influence a common agents' environment rather than agents in person).
An important example of the nonlocal vector field is 
\begin{equation}
  V_{t}(x,\mu,u) = f_{t}(x,u) + \int K_{t}(x-y,u)\d\mu(y),\label{VF-conv}
\end{equation}
where \( f \) models an external force pushing the agents and \( K \) stands for their internal interaction. Typical terminal cost functionals are
\begin{align*}
  \ell_{1}(\mu) = \int l(x)\d\mu(x) + \iint W(x,y) \d \mu(x)\d \mu(y), \quad
  \ell_{2}(\mu) = \frac{1}{2}\left\lvert\int x\d \mu(x)-m_{T}\right\rvert^{2}.
\end{align*}
Here, \( \ell_{1} \) represents the potential ($l$) and interaction ($W$) energy terms, while  \( \ell_{2} \) is related to the \emph{averaged control problem}~\cite{Zuazua2014}, where the goal is to bring the expectation of the distribution $\mu$ to some target position $m_T$. Finally, common versions of running cost term are
\begin{displaymath}
  L_{1}(t,\mu,u) = \frac{1}{2}\vert u\vert^{2},\quad
  L_{2}(t,\mu,u) = \frac{1}{2}\left\lvert \int x\d\mu(x) -m(t) \right\rvert^{2}.
\end{displaymath}
\( L_1 \) represents the ``total energy'' of the control action, and \( L_2 \) captures the problem of following a desired path \( t\mapsto m(t) \).

\section{Preliminaries}
\label{sec:prelim}

In this section, we introduce some notations, 
and recall several useful facts from analysis on the metric space of probability measures.

\subsection{Notation}

Throughout the paper, we use the following notation:

\begin{itemize}
  \item \( \lvert\cdot\rvert \) the Euclidean norm on \( \mathbb{R}^{n} \).
  \item \( \bm B_r\subset \mathbb{R}^{n} \) the closed unit ball of radius \( r \) centered at the origin.
  \item \( f_{\sharp}\mu \) pushforward measure for \( \mu\in \mathcal P(\mathbb{R}^n) \) and a Borel function \( f\colon \mathbb{R}^n\to \mathbb{R}^m \).
  
  \item \(\spt \mu\) the support of a measure $\mu$.
  
  \item \( \M^{m,n} \) the space of matrices \( A \) with \( m \) rows and \( n \) columns. 
  \item $x =
        \begin{pmatrix}
          x^{1}\\
          \vdots\\
          x^{n}
        \end{pmatrix}$ an  \( n \)-dimensional column vector, i.e., \( x\in \M^{n,1} = \R^{n}\).
  \item $p =
        \begin{pmatrix}
          p_{1}&\cdots&p_{n}
        \end{pmatrix}$ an \( n \)-dimensional row vector, i.e., \( p \in \M^{1,n}=(\R^{n})^{*} \).
  \item A vector field \( f \) on \( \R^{n} \) is a family of \( n \) real-valued functions \( f^{i}=f^{i}(t,x) \), \( i=1,\ldots,n \).
  \item A vector field \( f \) on \( \R^{n}\times(\R^{n})^{*} \) is a family of \( 2n \) real-valued functions \( f^{i}=f^{i}(t,x,p) \), \( f_{i}=f_{i}(t,x,p) \), \( i=1,\ldots,n \).
  \item \( \div_x f = \sum_{i=1}^{n} \partial_{x^{i}}f^{i} \) divergence of the vector field \( f=f(t,x) \) in \( x \).
  \item \( \div_{(x,p)} f=\sum_{i=1}^{n}\left(\partial_{x^{i}}f^{i}+\partial_{p_{i}}f_{i}\right) \) divergence of the vector field \( f=f(t,x,p) \) in \( (x,p) \).
  \item $D_xf =
        \begin{pmatrix}
          \partial_{x^{1}}f^{1}&\cdots&\partial_{x^{n}}f^{1}\\
          \vdots&\ddots&\vdots\\
          \partial_{x^{1}}f^{n}&\cdots&\partial_{x^{n}}f^{n}
        \end{pmatrix}$ derivative of the vector field \( f=f(t,x) \) in \( x \).
  \item $\nabla_{x}\psi =
        \begin{pmatrix}
          \partial_{x^{1}}\psi&\cdots&\partial_{x^{n}}\psi
        \end{pmatrix}$ gradient of a real-valued function \( \psi=\psi(t,x,p) \) in \( x \).
   \item $\nabla_{p}\psi =
        \begin{pmatrix}
          \partial_{p_{1}}\psi\\
          \vdots\\
          \partial_{p_{n}}\psi
        \end{pmatrix}$ gradient of a real-valued function \( \psi=\psi(t,x,p) \) in \( p \). 
\end{itemize}

Below, we will also deal with vector measures whose values belong to \( \M^{1,n} \), i.e., $\nu =
\begin{pmatrix}
  \nu_{1}&\cdots&\nu_{n}
\end{pmatrix}$, where \( \nu_{1},\ldots,\nu_{n} \) are Radon measures on \( \R^{n} \).
Given \( \phi\colon \R^{n}\to \R^{n} \), we set
$
 \displaystyle \langle\nu,\phi\rangle = \int \phi\cdot\d\nu \doteq \sum_{i=1}^{n} \int \phi^{i}\d\nu_{i}$.

Let \( X \) be a Polish space.
From measures on \( X \) one can construct several important topological spaces:
\(
\mathcal M(X)\supset \mathcal P(X) \supset \mathcal P_{2}(X) \supset \mathcal P_{c}(X).
\)
Here \( \mathcal{M}(X) \) consists of all signed Radon measures, \( \mathcal P(X) \) of all probability measures, \( \mathcal P_{2}(X) \) of all probability measures with finite second moments, \( \mathcal P_{c}(X) \) of all compactly supported probability measures.
Below, the Wasserstein distance~\cite{AGS} on \( \mathcal P_{2}(X) \) is always denoted by \( W_{2} \).

Given a Radon measure \( \mu \) on \( \mathbb{R}^n \), denote by \( L_{\mu}^{p}(\mathbb{R}^{n};\mathbb{R}^{m}) \) the space of all \( \mu \)-measurable maps (equivalence classes) \( f\colon\R^{n}\to\R^{m} \) such that {\( \|f\|_{L^p_\mu}\doteq\left(\int\vert f \vert^{p}\d \mu\right)^{1/p}<\infty \)}.
If \( \mu \) is the \( n \)-dimensional Lebesgue measure \( \mathcal L^{n} \), we simply write \( L^p(\mathbb{R}^{n};\mathbb{R}^{m}) \).

\subsection{The space $\mathcal{P}_{c}(\mathbb R^{n})$ and functions of probability measures}
\label{subsec:Pc}

The role of the main arena of our paper will be played by the space \( \mathcal{P}_{c}(\R^{n}) \) endowed with the so-called \emph{final topology}.
\begin{definition}
  \label{def:final}
  Let $(\mathcal X_n, \tau_n)$ be a sequence of topological spaces such that $\mathcal X_n\subset \mathcal X_{n+1}$ with continuous inclusion on every $n$.
  Let $\mathcal X = \cup_n\mathcal X_n$.
  The \emph{final topology} is the strongest topology $\tau$ on $\mathcal X$ which lets the inclusions $\mathrm{id}_n\colon \mathcal X_n\to \mathcal X$ be continuous for every $n$.
\end{definition}

In our case, $(\mathcal X_n,\tau_n) \doteq (\mathcal P(\bm B_n),W_2)$ and $\mathcal X \doteq \mathcal P_c(\mathbb R^n)$.
The final topology \( \tau \) on \( \mathcal{P}_{c}(\R^{n}) \) enjoys the following properties~\cite{GigliThesis}:
\begin{itemize}
  \item $\mu_n \xrightarrow{\tau} \mu$ if and only if $\mu_n\xrightarrow{W_{2}} \mu$ in $\mathcal{P}(\bm B_N)$ for some $N$,
  \item if $\mathcal K \subset \mathcal P_c(\mathbb R^n)$ is compact, then $\mathcal K\subset \mathcal{P}(\bm B_N)$ for some $N$,
  \item $\tau$ is a Hausdorff topology but it is not induced by any distance.
\end{itemize}

Below, we will constantly deal with mappings \( \Phi\colon I\times\R^{n}\times\P_{c}(\R^{n})\to \R^{m} \) of a particular regularity. Recall the respective 
\begin{definition}\label{def:regularity}
Let \( \Phi\) be a map \(I\times\R^{n}\times\P_{c}(\R^{n})\to \R^{m} \). We say that
\begin{enumerate}
  \item \( \Phi \) is a \emph{Carath\'eodory map} if and only if \( t\mapsto \Phi(t,x,\mu) \) is measurable for each \( (x,\mu) \), and \( (x,\mu)\mapsto \Phi(t,x,\mu) \) is sequentially continuous for each \( t \).
  \item \( \Phi \) is \emph{locally bounded} if its restriction on any compact subset of \( I\times\R^{n}\times\mathcal{P}_{c}(\R^{n}) \) is bounded.
  \item \( \Phi \) is \emph{locally Lipschitz} if and only if, for each \( t \), the restriction of \( (x,\mu)\mapsto\Phi(t,x,\mu) \) to any compact set \( \mathcal{K} \subset \R^{n}\times \mathcal{P}_{c}(\R^{n}) \) is Lipschitz with some constant \( L_{\mathcal{K}} \), independent of \( t \).
  \item \( \Phi \) is \emph{sublinear} if and only if there exists \( C>0 \) such that \( \left\lvert\Phi(t,x,\mu)\right\rvert\leq C \left( 1+\vert x\vert\right) \) for all \( t\), \( x \), \( \mu \).
\end{enumerate}
\end{definition}
Thanks to the outlined properties of the final topology, the definitions of the local boundedness and local Lipschitzianity can be given in the following equivalent way:
\begin{enumerate}
  \item[2\( ' \).]  \( \Phi \) is \emph{locally bounded} if and only if, for any compact \( \Omega\subset \R^{n} \), there exists \( C_{\Omega}>0 \) such that \( \left\lvert\Phi(t,x,\mu)\right\rvert\leq C_{\Omega} \) for all \( t\in I \), \( x\in \Omega \), \( \mu\in \P(\Omega) \);
  \item[3\( ' \).] \( \Phi \) is \emph{locally Lipschitz} if and only if, for any compact \( \Omega\subset \R^{n} \), there exists \( L_{\Omega}>0 \) such that \( \left\lvert\Phi(t,x,\mu) - \Phi(t,x',\mu')\right\rvert\leq L_{\Omega} \left( \vert x-x'\vert  + W_{2}(\mu,\mu') \right) \) for all \( t\in I \), \( x,x'\in \Omega \), \( \mu,\mu'\in \P(\Omega) \).
\end{enumerate}

\subsection{Derivatives in the space of probability measures}
\label{subsec:deriv}

There are several concepts of derivative of a function $\mathcal P \to \R$.  In this paper, we shall employ the notion of ``intrinsic derivative''~\cite{CardMaster2019}.

\begin{definition}[$\C^{1}$ maps]
  \label{def:flat}
  A function \( F \colon \P_{c}(\R^{n})\to \R \) is said to be \emph{of class \( \C^{1} \)} if and only if there exists a sequentially continuous, locally bounded map \( \frac{\delta F}{\delta\mu}\colon {\P}_{c}(\R^{n})\times\R^{n}\to\R \) such that
  \begin{gather*}
    F(\mu') - F(\mu) = \int_{0}^{1}\int\frac{\delta F}{\delta\mu}\left((1-t)\mu+t\mu',y\right)\d(\mu'-\mu)(y)\d t\qquad \forall \mu,\mu'\in \P_{c}(\R^{n}).
  \end{gather*}
  Since \( \frac{\delta F}{\delta\mu} \) is defined up to an additive constant, we adopt the normalization convention
  \begin{displaymath}
    \int \frac{\delta F}{\delta\mu}(\mu,y)\d\mu(y) = 0\qquad \forall \mu\in \P_{c}(\R^{n}).
  \end{displaymath}
\end{definition}

\begin{definition}
  \label{def:intrinsic}
  Let \( \frac{\delta F}{\delta\mu} \) be \( \C^{1} \) in \( y \).
  Then the \emph{intrinsic derivative} \( D_{\mu}F\colon \mathcal{P}_{c}(\R^{n})\times\R^{n}\to\R^{n} \) is defined by \( D_{\mu}F \doteq D_{y}\frac{\delta F}{\delta\mu} \).
\end{definition}

Some important properties of the intrinsic derivative are gathered in the following proposition, {which combines the statements of Propositions 2.2-2.4 from~\cite{cardaliaguetAnalysisSpaceMeasures2019}}.
\begin{proposition}
  \label{prop:flat}
  Let \( F\colon\P_{c}(\R^{n})\to \R \) be \( \C^{1} \), \( \frac{\delta F}{\delta\mu} \) be \( \C^{1} \) in \( y \), and \( D_{\mu}F \) be sequentially continuous and locally bounded.
  Then, the following holds:
  \begin{enumerate}
    \item For any Borel measurable, locally bounded map \( \phi\colon \R^{n}\to \R^{n} \), the function \( s\mapsto F\left((\id+s\phi)_{\sharp}\mu\right) \) is differentiable at zero, and
    \begin{equation}
      \frac{d}{ds}\Big\vert_{s=0}F\left((\id+s\phi)_{\sharp}\mu\right) = \int D_{\mu}F(\mu,y)\cdot\phi(y)\d\mu(y).\label{intrinsic}
    \end{equation}
    \item Given a compact set \( \Omega\subset\mathbb R^n \), the restriction of \( F \) to \( \P(\Omega) \) satisfies
    \begin{multline*}
      \left\lvert F(\mu') - F(\mu) - \iint D_{\mu}F(\mu,y)\cdot (y-x)\d\Pi(x,y) \right\rvert
      \leq
      o \left( \left( \iint \vert x-y\vert ^{2}\d\Pi(x,y) \right)^{1/2} \right),
    \end{multline*}
    for any \( \mu,\mu'\in \P(\Omega) \) and any transport plan \( \Pi \) between \( \mu \) and \( \mu' \).
    \item The quantity \( \frac{\delta F}{\delta \mu} \) can be calculated as follows:
     \begin{displaymath}
       \frac{\delta F}{\delta\mu}(\mu,y) = \lim_{h\to 0+}\frac{1}{h}\left(F\left((1-h)\mu+h\delta_{y}\right)-F(\mu)\right).
     \end{displaymath}
  \end{enumerate}
\end{proposition}

The first property links the intrinsic derivative with a ``directional'' derivative, where \( \phi \) plays the role of direction. The second one relates the notion of intrinsic derivative with the so-called \emph{localized Wasserstein derivative}~\cite{BonnetFrankowska2021a}:
\begin{definition}[localized Wasserstein derivative]
  \label{def:W}
  We say that \( F\colon \P_{2}(\R^{n})\to \R \) is \emph{locally differentiable} at \( \mu\in \P_{2}(\R^{n}) \) if there exists a tangent vector \( \xi\in \Tan_{\mu}{\P}_{2}(\R^{n}) \) such that, for any compact set \( \Omega \supset \spt \mu\), the restriction of \( F \) to \( \P(\Omega) \) satisfies
  \begin{displaymath}
    F(\mu') - F(\mu) = \iint \langle\xi(x),y-x\rangle\d\Pi(x,y) + o\left(\left(\iint \vert x-y\vert ^{2}\d\Pi(x,y)\right)^{1/2}\right),
  \end{displaymath}
  for any \( \mu'\in \mathcal P(\Omega) \) and any transport plan \( \Pi \) between \( \mu \) and \( \mu' \). Such \( \xi \) is uniquely defined and called the \emph{localized Wasserstein derivative} of \( F \) at \( \mu \). 
\end{definition}
Recall that the tangent space $\Tan_{\mu}{\P}_{2}(\R^{n})$ to ${\P}_{2}(\R^{n})$ at \( \mu\in \P_{2}(\R^{n}) \) is introduced as
\begin{displaymath}
  \Tan_{\mu}{\P}_{2}(\R^{n}) = \overline{\left\{\nabla\phi\;\text{s.t.}\;\phi\in \C^{\infty}_{c}(\R^{n})\right\}}^{L^{2}_{\mu}} \subset L^{2}_{\mu} = L^{2}_{\mu}(\R^{n};\R^{n}).
\end{displaymath}

Proposition~\ref{prop:flat} says that any \( \C^{1} \) functional on the space of probability measures with sequentially continuous and locally bounded intrinsic derivative \( D_{\mu}F \) is locally differentiable at any \( \mu\in \P_{c}(\R^{n}) \), and the projection of \( D_{\mu}F(\mu,\cdot) \) onto \( \Tan_{\mu}{\P}_{2}(\R^{n}) \) coincides with the corresponding localized Wasserstein derivative.

The third assertion of Proposition~\ref{prop:flat} offers a convenient tool for practical calculation of the intrinsic derivative. We illustrate this machinery with the use of the following  paradigmatic example. 
\begin{example}
  \label{ex:conv}
  Let \( K\colon\R^{n}\to\R^{n} \) be a \( \C^{1} \) map.
  Fixed \( x\in \R^{n} \), let us compute the intrinsic derivative of the functional
 $\displaystyle \mu \mapsto F(x,\mu) \doteq (K*\mu)(x) \doteq  \int K(x-y)\d\mu(y).$
By observing that
\begin{displaymath}
  F\left(x,(1-t)\mu+t\delta_{y}\right) = (1-t)\int K(x-y)\d\mu(y) + t K(x-y),
\end{displaymath}
the flat derivative is easily found as 
\begin{displaymath}
  \frac{\delta F}{\delta\mu}(x,\mu,y) = K(x-y) -\int K(x-z)\d\mu(z),
\end{displaymath}
which gives:
 \(
 D_{\mu}F(x,\mu,y) = D_{y}\frac{\delta F}{\delta\mu}(x,\mu,y) = -DK(x-y).
 \)

\end{example}

Recall another useful fact:
\begin{lemma}
  \label{lem:C1_lip}
  Let \( F \) be the same as in Proposition~\ref{prop:flat}.
  Then \( F \) is locally Lipschitz.
\end{lemma}
\begin{proof}
  Fix a compact set \( \Omega \) and two measures \( \mu,\mu'\in \P(\Omega) \).
  Denote by \( \Pi \) an optimal plan between \( \mu \) and \( \mu' \) and let \( \mu_{t}=(1-t)\mu+t\mu' \).
  Then, we have
  \begin{displaymath}
    F(\mu') - F(\mu) = \int_{0}^{1} \iint\left[\frac{\delta F}{\delta\mu}\left(\mu_{t},y\right) - \frac{\delta F}{\delta\mu}\left(\mu_{t},x\right)\right]\d\Pi(x,y)\d t.
  \end{displaymath}
  The difference in the squared brackets is
  \begin{displaymath}
    \int_{0}^{1}D_{y}\frac{\delta F}{\delta\mu}\left(\mu_{t},(1-s)y + sx\right)(y-x)\d s.
  \end{displaymath}
  Hence the statement follows from the local boundedness of \( D_{\mu}F = D_{y}\frac{\delta F}{\delta\mu} \).
\end{proof}

\begin{definition}
  We say that \( F\colon \P_{c}(\R^{n})\to\R \) is \emph{of class \( \C^{1,1} \)} if \( F \) is \( \C^{1} \), \( \frac{\delta F}{\delta\mu} \) is \( \C^{1} \) in \( y \), and the intrinsic derivative \( D_{\mu}F \) is locally Lipschitz and locally bounded.
\end{definition}

\subsection{Nonlocal vector fields and their flows}
\label{subsec:vfields}

A time-dependent nonlocal vector field is a map \(V\colon I\times\R^{n}\times\P_{c}(\R^{n})\to \R^{n} \). If the dependence on $\mu \in \P_{c}(\R^{n})$ is fictitious  we say that $V$ is a local vector field (or simply ``vector field''). The basic regularity of nonlocal vector fields is understood in the sense of Definition~\ref{def:regularity}. 

It is well-known that the local transport PDEs can be studied using their characteristic flows. Recall the following 

\begin{definition}
  \label{def:vf}
  We say that a vector field \( v\colon I\times \R^{n}\to\R^{n} \) is \emph{of class \( \C^{1,1} \)} if
  \begin{enumerate}
    \item \( v \) is a locally bounded Carath\'eodory map;
    \item \( v \) is \( \C^{1} \) in \( x \) for each \( t \);
    \item \( D_{x}v\colon I\times \R^{n}\to \M^{n,n} \) is Carath\'eodory, locally bounded and locally Lipschitz.
  \end{enumerate}
\end{definition}

Any \emph{sublinear \( \mathcal C^{1,1} \) vector field} \( v \) generates a unique continuous map \( P\colon I\times I\times\R^{n}\to\R^{n} \) named the \emph{flow} of \( v_{t} \); this map is defined such that, for each \( t_{0}\in I \) and \( x\in \R^{n} \), $t \mapsto P_{t_{0},t}(x)$ is as a solution of the Cauchy problem
\begin{displaymath}
  \partial_t P_{t_{0},t}(x) = v_{t}\left(P_{t_{0},t}(x)\right),\quad P_{t_{0},t_{0}}(x) = x.
\end{displaymath}
For any \( t_{0},t\in I \) the map \( P_{t_{0},t}\colon\R^{n}\to\R^{n} \) is a \( \C^{2} \) diffeomorphism.
Moreover, it satisfies the semigroup property:
  $P_{t_{1},t_{2}}\circ P_{t_{0},t_{1}} = P_{t_{0},t_{2}}$ for all $t_{0},t_{1},t_{2}\in I.$

 In fact, the concept of flow can be extended to the case of nonlocal vector fields. To this end, we modify Definition~\ref{def:vf} as follows:
\begin{definition}
  \label{def:nvf}
  We say that a \emph{nonlocal} \emph{vector field} \( V\colon I\times \R^{n}\times \P_{2}(\R^{n})\to\R^{n} \) is \emph{of class \( \C^{1,1} \)} if
  \begin{enumerate}[1)]
    \item \( V \) is a locally bounded Carath\'eodory map;
    \item \( V \) is \( \C^{1} \) in \( x \) for each \( t \) and \( \mu \), and \( \C^{1} \) in \( \mu \) 
          for each \( t \) and \( x \);
    \item both \( D_{x}V\colon I\times\R^{n}\times\mathcal{P}_{c}(\R^{n})\to \M^{n,n} \) and \( D_{\mu}V\colon I\times \R^{n}\times \mathcal{P}_{c}(\R^{n})\times\R^{n}\to \M^{n,n} \) are Carath\'eodory, locally bounded and locally Lipschitz.
  \end{enumerate}
\end{definition}
Now, observe that any \emph{sublinear \( \C^{1,1} \) nonlocal vector field} \( V \) generates a unique sequentially continuous function \( X\colon I \times I\times\R^{n}\times \P_{c}(\R^{n})\to\R^{n} \) such that, for each \( x\in\R^{n} \) and \( \vartheta\in\P_{2}(\R^{n}) \), $t \mapsto X^{\vartheta}_{t_0,t}(x)$ is a solution of the ODE
\begin{displaymath}
  \partial_t X^{\vartheta}_{t_0,t}(x) = V_{t}\left(X^{\vartheta}_{t_0,t}(x),X^{\vartheta}_{t_0,t\sharp}\vartheta\right),\quad X^{\vartheta}_{t_0,t_0}(x) = x.
\end{displaymath}

We abbreviate $X^\vartheta_t = X^\vartheta_{0,t}$ and stress that \( \mu_{t}=(X^{\vartheta}_{t})_\sharp\vartheta \) is the unique solution of the nonlocal continuity equation
\begin{displaymath}
  \partial_{t}\mu_{t} + \div_x\left(V_{t}(\cdot,\mu_{t})\mu_{t}\right) = 0,\quad \mu_{0} = \vartheta.
\end{displaymath}
We call the map \( X\) the \emph{flow of the nonlocal vector field} \( V \). 

Notice that, for a given \( \vartheta \), we can define \( v_{t}(x)\doteq V_{t}(x,X^{\vartheta}_{t\sharp}\vartheta) \) and denote by \( P \) the flow of \( v \).
It is clear that \( X^{\vartheta}_{0,t} = P_{0,t} \).
We will use this fact below several times.

{The outlined facts (existence of the flow, well-posedness of the nonlocal continuity equation, and the representation formula for its solution) are well-known, refer, e.g., to 
~\cite{BonnetRossi2019, Fornasier2014, PiccoliRossi2013}.
}

\subsection{$\Oloc(\lambda^{2})$ families of vector fields}

In this section, we discuss some differential properties of nonlocal vector fields and their flows.
\begin{definition}
  \label{def:littleo2}
  Let \( \Phi^{\lambda}\colon \mathcal{X} \mapsto \R^{m} \), \( \lambda\in [0,1] \), be a family of functions on a topological space \( \mathcal{X} \).
  We say that \( \Phi^{\lambda} \) is \emph{\( \Oloc(\lambda^{2}) \) family} and write
  \begin{displaymath}
   \Phi^{\lambda} = \Oloc(\lambda^{2})\quad\text{or}\quad
   \Phi^{\lambda}(x) = \Oloc(x;\lambda^{2})
 \end{displaymath}
 if for any compact set \(\mathcal{K} \subset \mathcal{X} \) there exists \( C_{\mathcal{K}}>0 \) such that \( \left\lvert\Phi^{\lambda}(x)\right\rvert\leq C_{\mathcal{K}}\lambda^{2} \) for all \( \lambda\in [0,1] \) and \( x\in \mathcal{K} \).
\end{definition}

In particular, a family of nonlocal vector fields \( V^{\lambda}\colon I\times \R^{n}\times \mathcal{P}_{c}(\R^{n})\to\R^{n}\) is \( \Oloc(\lambda^{2}) \) if, for any compact \( \Omega\subset \R^{n} \), there exists \( C_{\Omega}>0 \) such that \( \left\vert V^{\lambda}_{t}(x,\mu)\right\vert\leq C_{\Omega}\lambda^{2} \), for all \( \lambda\in [0,1] \), \( t\in I \), \( x\in \Omega \), \( \mu\in \P(\Omega) \).

\begin{lemma}
  \label{lem:O2}
  Let \( V \) be a nonlocal vector field of class \( \C^{1,1} \).
  Then, for any locally bounded \( \phi,\psi\colon\R^{n}\to \R \), one has
  \begin{multline*}
    V_{t} \left( x+\lambda\phi(x), (\id+\lambda\psi)_{\sharp}\mu  \right) - V_{t}(x,\mu) \\ - \lambda D_{x}V_{t}(x,\mu)\phi(x) - \lambda\int D_{\mu}V_{t}(x,\mu,y)\psi(y)\d\mu(y) = \Oloc(t,x,\mu;\lambda^{2}),\quad {\lambda\in[0,1]}.
  \end{multline*}
  Moreover, the constant \( C_{\Omega} \) that guaranties the estimate
  \begin{displaymath}
    \Oloc(t,x,\mu;\lambda^{2})\leq C_{\Omega}\lambda^{2}\qquad \forall (t,x,\mu)\in I\times\Omega\times\P(\Omega)
  \end{displaymath}
  depends only on the data
  \begin{equation}
    \label{eq:r}
  r \doteq \max\left\{\vert x\vert +\max\{\vert\phi(x)\vert,\vert\psi(x)\vert\}\;\colon\;x\in \Omega\right\}
\end{equation}
\begin{displaymath}
 L_{r} =\max_{t \in I}\left\{\max_{\bm B_{r}\times\P(\bm B_{r})}\lip(D_{x}V_{t}), \max_{\bm B_{r}\times\P(\bm B_{r})\times \bm B_{r}}\lip(D_{\mu}V_{t})\right\}.
\end{displaymath}
\end{lemma}
\begin{proof}
  We split the proof into several steps.

  1. Fix a compact set \( \Omega \) and a triple \( (t,x,\mu)\in I\times\Omega\times\mathcal{P}(\Omega) \). Consider the identity:
  \begin{align*}
    V_{t} \left( x+\lambda\phi(x), (\id+\lambda\psi)_{\sharp}\mu  \right) &- V_{t}(x,\mu)\\
    &=V_{t} \left( x+\lambda\phi(x), (\id+\lambda\psi)_{\sharp}\mu  \right) - V_{t} \left( x, (\id+\lambda\psi)_{\sharp}\mu  \right)\\
    &+ V_{t} \left( x, (\id+\lambda\psi)_{\sharp}\mu  \right) - V_{t}(x,\mu).
  \end{align*}
  By the mean value theorem, the first difference in the right-hand side takes the form
  \begin{displaymath}
    \lambda\int_{0}^{1}D_{x}V_{t} \left( x + s\lambda\phi(x), (\id+ \lambda\psi)_{\sharp}\mu\right)\phi(x)\d s,
  \end{displaymath}
  and the second one yields
  \begin{align*}
    \int_{0}^{1}
    &\int\frac{\delta V_{t}}{\delta\mu} \left( \mu_{\lambda,\tau},y \right)\d \left( (\id+\lambda\psi)_{\sharp}\mu - \mu \right)(y)\d\tau\\
    &=\int_{0}^{1}\int\left[\frac{\delta V_{t}}{\delta\mu} \left( \mu_{\lambda,\tau},y+\lambda\psi(y) \right) - \frac{\delta V_{t}}{\delta\mu} \left( \mu_{\lambda,\tau},y \right)\right]\d \mu(y)\d\tau\\
    &=\lambda\int_{0}^{1}\int_{0}^{1}\int D_{\mu}V_{t} \left(x, \mu_{\lambda,\tau}, y+s\lambda\psi(y)\right)\psi(y)\d\mu(y)\d s\d\tau,
  \end{align*}
  where \( \mu_{\lambda,\tau} = (1-\tau)\mu+\tau(\id+\lambda\psi)_{\sharp}\mu \).

  2. Let \( r \) be as in~\eqref{eq:r}.
  Then \( x +\lambda\phi(x) \in \bm B_{r} \) and
  \( (\id+\lambda\psi)_{\sharp}\mu\in \P(\bm B_{r}) \) for all \( \lambda\in[0,1] \), \( x\in \Omega \), \( \mu\in \P(\Omega) \).
  Since \( D_{x}V \) and \( D_{\mu}V \) are locally Lipschitz,
  {
  \begin{align}
    &\left\lvert D_{x}V_{t} \left( x + s\lambda\phi(x), (\id+ \lambda\psi)_{\sharp}\mu\right)\phi(x) - D_{x}V_{t}(x,\mu)\phi(x)\right\rvert\notag\\
    &\hspace{160pt}\leq \lambda L_{r} \left(\left\lvert \phi(x) \right\rvert^{2}+ \|\psi\|_{L^{2}_{\mu}}\vert\phi(x)\vert\right)\label{eq:aux1}\\
    &\left\vert\int D_{\mu}V_{t} \left(x, \mu_{\lambda,\tau}, y+s\lambda\psi(y)\right)\psi(y)\d\mu(y)-\int D_{\mu}V_{t}(x,\mu,y)\psi(y)\d\mu(y) \right\vert\notag\\
    &\hspace{160pt}\leq L_{r} \left( W_{2}(\mu,\mu_{\lambda,\tau})\|\psi\|_{L^{1}_{\mu}} +\lambda \|\psi\|^{2}_{L^{2}_{\mu}}  \right).\label{eq:aux2}
  \end{align}
}
  3. Let us estimate \(W_{2} \left( \mu,\mu_{\lambda,\tau} \right) \).
  To this end, recall that
  \begin{equation}
    \label{eq:convex}
    W_{2}^{2}\left((1- \tau)\mu_{0}+\tau\mu_{1},\nu\right)\leq (1-\tau)W_{2}^{2}(\mu_{0},\nu) + \tau W_{2}^{2}(\mu_{1},\nu),
  \end{equation}
  for all \( \mu_{0},\mu_{1},\nu \in \P_{2}(\R^{n})\) and all \( \tau\in [0,1] \).
  This inequality becomes evident if we note that, for any \( \Pi_{0}\in \Gamma_{o}(\mu_{0},\nu) \) and \( \Pi_{1}\in \Gamma_{o}(\mu_{1},\nu) \), the convex combination \( (1-\tau)\Pi_{0}+\tau\Pi_{1} \) is a transport plan between \( (1- \tau)\mu_{0}+\tau\mu_{1}\) and \( \nu \).
  In our case,~\eqref{eq:convex} implies that
  \begin{displaymath}
    W_{2} \left( \mu,\mu_{\lambda,\tau} \right)\leq \sqrt{\tau}W_{2} \left( \mu, (\id+\lambda\psi)_{\sharp}\mu \right)\leq \sqrt{\tau}\lambda\|\psi\|_{L^{2}_{\mu}}.
  \end{displaymath}
  The statement now follows from~\eqref{eq:aux1}, \eqref{eq:aux2} and the inequalities \( \vert\phi\vert\leq r \), \( \vert\psi\vert\leq r \) on \( \Omega \).
\end{proof}

Arguments, similar to those of the previous proof, lead to the following slight modification of Lemma~\ref{lem:O2}.
\begin{lemma}
  \label{lem:O2-mod}
  Let \( V \) be a nonlocal vector field of class \( \C^{1,1} \) and \( X\colon I\times\R^{n}\times\mathcal{P}_{c}(\R^{n})\to \R^{n}\) be a sequentially continuous and locally bounded map such that \( x\mapsto X^{\mu}_{t}(x)\) is bijective for all \( t \) and \( \mu \).
  Then, for any locally bounded Carath\'eodory maps \( \phi,\psi\colon I\times\R^{n}\times\P_{c}(\R^{n})\to\R^{n} \), we have
  \begin{align*}
    V_{t} & \left( X_{t}^{\mu}(x) +\lambda\phi^{\mu}_{t}(x), (X_{t}^{\mu}+\lambda\psi^{\mu}_{t})_{\sharp}\mu  \right) - V_{t}\left(X_{t}^{\mu}(x),X_{t\sharp}^{\mu}\mu\right)\\
     &- \lambda D_{x}V_{t}\left(X_{t}^{\mu}(x),X_{t\sharp}^{\mu}\mu\right)\phi^{\mu}_{t}(x)
     - \lambda\int D_{\mu}V_{t}\left(X_{t}^{\mu}(x),X_{t\sharp}^{\mu}\mu,X_{t}^{\mu}(y)\right)\psi^{\mu}_{t}(y)\d\mu(y)\\
     & = \Oloc(t,x,\mu;\lambda^{2}).
  \end{align*}
  Moreover, the constant \( C_{\Omega} \) which guaranties the estimate
  \begin{displaymath}
    \Oloc(t,x,\mu;\lambda^{2})\leq C_{\Omega}\lambda^{2}\qquad \forall (t,x,\mu)\in I\times\Omega\times\P(\Omega)
  \end{displaymath}
  depends only on
  \begin{equation}
    \label{eq:r2}
  r \doteq \max\left\{\left\vert X_{t}^{\mu}(x)\right\vert+\max\{\vert\phi^{\mu}_{t}(x)\vert,\vert\psi^{\mu}_{t}(x)\vert\}\;\colon\;(t,x,\mu)\in I\times\Omega\times\P(\Omega)\right\}
\end{equation}
and \( L_{r} \) which bounds, for all \( t \), the Lipschitz constants of \( D_{x}V_{t} \) and \( D_{\mu}V_{t} \):
\begin{displaymath}
  \lip(D_{x}V_{t})\leq L_{r}\text{ on } \bm B_{r}\times\P(\bm B_{r}),\quad \lip(D_{\mu}V_{t})\leq L_{r} \text{ on } \bm B_{r}\times\P(\bm B_{r})\times \bm B_{r}.
\end{displaymath}
\end{lemma}

The following 
presents a refined version of the formula \eqref{intrinsic} for the intrinsic derivative. 
\begin{lemma}
  \label{lem:O2-mod2}
  Let \( F\colon \P_{c}(\R^{n})\to \R \) be of class \( \C^{1,1} \), and \( \Phi^{\lambda}\colon \R^{n}\to\R^{n} \), \( \lambda\in [0,1] \), be a family of Borel maps which can be expanded as follows:
  \begin{equation}
    \label{eq:Phi_expansion}
    \Phi^{\lambda}(x) = \Phi^{0}(x) + \lambda \phi(x) + \Oloc(x;\lambda^{2}),
  \end{equation}
  for some \( \phi\colon \mathbb{R}^n\to \mathbb{R}^n \).
  Then,
  \begin{displaymath}
    F\left(\Phi^{\lambda}_{\sharp}\mu \right) - F(\Phi^{0}_{\sharp}\mu) - \lambda \int D_{\mu}F\left(\Phi^{0}_{\sharp}\mu,\Phi^{0}(y)\right)\phi(y)\d\mu(y) = \Oloc(\mu;\lambda^{2}).
  \end{displaymath}
\end{lemma}
\begin{proof}
  In view of Lemma~\ref{lem:O2-mod}, it suffices to show that
  \begin{displaymath}
    F\left(\Phi^{\lambda}_{\sharp}\mu  \right) - F\left( (\Phi^{0}+\lambda\phi)_{\sharp}\mu \right) = \Oloc(\mu; \lambda^{2}).
  \end{displaymath}
  According to Lemma~\ref{lem:C1_lip}, \( F \) is locally Lipschitz.
  Hence, for any compact \( \Omega\subset \R^{n} \), there exists \( L_{\Omega}>0 \) such that
\begin{align*}
  \Big\vert F\left(\Phi^{\lambda}_{\sharp}\mu\right) - F\left((\Phi^{0}+\lambda\phi)_{\sharp}\mu\right)\Big\vert
  &\leq L_{\Omega}W_{2} \left( \Phi^{\lambda}_{\sharp}\mu, (\Phi^{0} +\lambda\phi)_{\sharp}\mu \right)\\
  &\leq L_{\Omega}\left\| \Phi^{\lambda} - \Phi^0 - \lambda\phi\right\|_{L^{2}_{\mu}},
\end{align*}
for all \( \mu\in \P(\Omega) \).
Now, by using~\eqref{eq:Phi_expansion}, we complete the proof.
\end{proof}

\subsection{Derivative of the flow}
\label{subsec:flow_deriv}

Recall that, for a fixed initial measure, any sublinear \( \C^{1,1} \) nonlocal vector field (n.v.f.) \( V \) generates a map \( X \) that can be thought of as its flow.
We shall study the flow \( X^{\lambda} \) of the perturbed n.v.f. \( V^{\lambda} =  V + \lambda W \), where \( W \) is also sublinear \( \C^{1,1} \), and \( \lambda\in [0,1] \).

{The results of this section, which provide the linearization of the nonlocal flow, are largely similar to those of \cite[sec. 3.2]{BonnetFrankowska2021a} (both in their statements and their proofs). 
However, in contrast to \cite{BonnetFrankowska2021a}, we accept here nonlocal perturbations of the vector field.
On the other hand, we impose slightly more restrictive assumptions, enabling us to expand the nonlocal flow up to the term of order $O(\lambda^2)$ rather than $o(\lambda)$ as demonstrated in \cite{BonnetFrankowska2021a}. This fact will play a crucial role in establishing the convergence of our numerical algorithm in Section~\ref{subsec:alg}.
}
\begin{theorem}
  \label{thm:dflow}
  Let \( V, W \) be sublinear \( \C^{1,1} \) nonlocal vector fields, \( X \) be the flow of \( V \) and \( X^{\lambda} \) be the flow of \( V^{\lambda}\doteq V+\lambda W \), where \( \lambda\in [0,1] \).
  Then
\[
X^{\lambda}- X - \lambda w = \Oloc(\lambda^{2}),
\]
  where \( w\colon I\times \mathbb{R}^n \times \mathcal P_c(\mathbb{R}^n)\to \mathbb{R}^n \) satisfies the differential equation
  \begin{align}
    \partial_tw^{\vartheta}_{t}\left(x\right)
    &=
    D_xV_{t}\left(X^{\vartheta}_{t}(x),X^{\vartheta}_{t\sharp}\vartheta\right)w^{\vartheta}_t(x) +
    \int D_{\mu}V_{t}\left(X^{\vartheta}_{t}(x),X^{\vartheta}_{t\sharp}\vartheta, X^{\vartheta}_{t}(y)\right)w^{\vartheta}_{t}\left(y\right)\d\vartheta(y) \notag \\
    &+ W_{t}\left(X^{\vartheta}_{t}(x) ,X^{\vartheta}_{t\sharp}\vartheta\right)
    \label{eq:w}
  \end{align}
  and the initial condition
  \begin{equation}
    \label{eq:w_init}
    w^{\vartheta}_{0}(x) = 0.
  \end{equation}
  Moreover, the constant \( C_{\Omega} \) which guaranties the estimate
  \begin{displaymath}
    \Oloc(t,x,\vartheta;\lambda^{2})\leq C_{\Omega}\lambda^{2}\qquad \forall (t,x,\vartheta)\in I\times\Omega\times\P(\Omega)
  \end{displaymath}
  depends only on the constants \( \rho \), \( C_{\rho} \), \( r \), \( L_{r} \) defined below by~\eqref{eq:rho}, \eqref{eq:Cbound}, \eqref{eq:Rr}, \eqref{eq:lip_bound}.
\end{theorem}

\begin{remark}
  First, notice that \( \vartheta \) in~\eqref{eq:w} can be considered as a parameter.
  Thus,~\eqref{eq:w} can be thought of as ``linear transport equation with nonlocal source term''.
  One can easily show (for example, by fixed-point arguments) that~\eqref{eq:w}, \eqref{eq:w_init} has a unique continuous solution \( w \) (see also~\cite{BonnetFrankowska2021a,BonnetRossi2019}, where such solution is constructed explicitly for the case \( W_{t}(x,\mu)\equiv W_{t}(x) \)).
  Moreover, \( w \) is sequentially continuous as a function of \( t \), \( x \), \( \vartheta \).
\end{remark}

Before presenting the proof, note that our assumptions on \( V \) and \( W \) imply that there exists \( C>0 \) such that \( \left\vert V^{\lambda}_{t}(x,\mu)\right\vert\leq C\left(1+\vert x\vert \right) \), for all \( t \), \( x \), \( \mu \), \( \lambda \).
This means that \( \left\vert X^{\lambda,\vartheta}_{t}(x) \right\vert\leq e^{Ct}(Ct+\vert x\vert ) \) for all \( t \), \( x \), \( \vartheta \), \( \lambda \).
As a consequence, $  (t,x,\vartheta)\mapsto \left( t,X^{\vartheta,\lambda}_{t}(x), X^{\vartheta,\lambda}_{t\sharp}\vartheta \right)$
maps \( I\times\Omega\times\P(\Omega)  \) into \( I\times\bm B_{\rho}\times\P(\bm B_{\rho})  \), where
\begin{equation}
  \label{eq:rho}
  \rho \doteq \max\left\{e^{CT}(CT+\vert x\vert )\colon\; x\in \Omega\right\}.
\end{equation}
Using the local boundedness of \( D_{x}V \), \( D_{\mu}V \) and \( W \), we can find \( C_{\rho}>0 \) such that
\begin{equation}
  \label{eq:Cbound}
  \vert D_{x}V\vert \leq C_{\rho}\quad \vert D_{\mu}V\vert\leq C_{\rho},\quad \vert W\vert\leq C_{\rho} \quad\text{on}\quad I\times \bm B_{\rho}\times \P(\bm B_{\rho}).
\end{equation}
Now, it follows from~\eqref{eq:w}, \eqref{eq:w_init} that
\begin{equation}
  \label{eq:w_bound}
  \vert w\vert \leq C_{\rho}e^{2C_{\rho}T}\quad \text{on}\quad I\times\Omega\times\P(\Omega).
\end{equation}
This implies that $  (t,x,\vartheta)\mapsto \left( t,(X^{\lambda,\vartheta}_{t}+\lambda w^{\vartheta}_{t})(x), (X^{\lambda,\vartheta}_{t}+\lambda w^{\vartheta}_{t})_{\sharp}\vartheta \right)$
maps \( I\times\Omega\times\P(\Omega)  \) into \( I\times\bm B_{r}\times\P(\bm B_{r})  \), where
\begin{equation}
  \label{eq:Rr}
  r\doteq \rho + C_{\rho}e^{2C_{\rho}T}.
\end{equation}
Finally, since \( V^{\lambda} \), \( D_{x}V^{\lambda} \) and \( D_{\mu}V^{\lambda} \) are locally Lipschitz, we choose \( L_{r}>0 \) such that
\begin{equation}
  \label{eq:lip_bound}
  \begin{matrix}
  \lip(V^{\lambda}_{t})\leq L_{r},\quad \lip(D_{x}V^{\lambda}_{t})\leq L_{r}\quad\text{on}\quad \bm B_{r}\times\P(\bm B_{r}),\\ \lip(D_{\mu}V^{\lambda}_{t})\leq L_{r}\quad\text{on}\quad\bm B_{r}\times \P(\bm B_{r})\times \bm B_{r},
  \end{matrix}
\end{equation}
for all \( t\in I \) and \( \lambda\in [0,1] \).

Fix a compact set \( \Omega\subset \R^{n} \) and a measure \( \vartheta\in \P(\Omega) \).
From now on, we will omit the index \( \vartheta \) in \( X^{\lambda,\vartheta}_{t} \) and \( w_{t}^{\vartheta} \).
Consider the following set:
\begin{displaymath}
  \mathcal{X}(\Omega) = \left\{\phi\in C^{0}(I\times\Omega;\R^{n})\colon \vert\phi_{t}(x)\vert\leq e^{Ct}(Ct+\vert x\vert)\right\},
\end{displaymath}
and equip it with the norm $ \|\phi\|_{\sigma} = \max_{I\times \Omega} e^{-\sigma t}\vert\phi_{t}(x)\vert$, $\sigma>0$. 
Since \( \|\cdot\|_{\sigma} \) is equivalent to the standard \( \sup \) norm, \( \mathcal{X}(\Omega) \) becomes a complete metric space.

Finally, for any \( \lambda\in [0,1] \) and \( \phi\in \mathcal{X}(\Omega) \), we define
\begin{displaymath}
  \mathcal{F}(\lambda,\phi)(t,x) = x + \int_{0}^{t}V_{\tau}^{\lambda}\left(\phi_{\tau}(x), \phi_{\tau\sharp}\vartheta\right)\d \tau,\quad
  t\in I,\; x\in \Omega.
\end{displaymath}
One can easily check that \( \mathcal{F} \) maps \( [0,1]\times \mathcal{X}(\Omega) \) to \( \mathcal{X}(\Omega) \).

\begin{lemma}
  \label{lem:contraction}
  The map \( \phi\mapsto \mathcal{F}(\lambda,\phi) \) is contractive in the \( \sigma \)-norm for all sufficiently large \( \sigma \). Moreover, the corresponding Lipschitz constant \( \kappa < 1 \) does not depend on \( \lambda \).
\end{lemma}

\begin{proof}
  Let \( r \) be defined by~\eqref{eq:Rr}.
  Given \( \phi,\psi\in \mathcal{X}(\Omega) \), we have
  \begin{align*}
    \left\vert  \mathcal{F}(\lambda,\phi) - \mathcal{F}(\lambda,\psi)\right\vert(t,x)
    &\leq \int_{0}^{t}\left\vert V^{\lambda}_{\tau}\left(\phi_{\tau}(x),\phi_{\tau\sharp}\vartheta\right) - V^{\lambda}_{\tau}\left(\psi_{\tau}(x),\psi_{\tau\sharp}\vartheta\right) \right\vert\d\tau\\
    &\leq L_{r}\int_{0}^{t}\left(\left\|\phi_{\tau}-\psi_{\tau}\right\|_{\C^{0}(\Omega;\R^{n})}+\left\|\phi_{\tau}-\psi_{\tau}\right\|_{L^{2}_{\vartheta}}\right)\d\tau,
  \end{align*}
  for any \( t\in I \), \( x\in \Omega \), \( \lambda\in [0,1] \).
  Since
  $\left\|\phi_{\tau}-\psi_{\tau}\right\|_{L^{2}_{\vartheta}}\leq \left\|\phi_{\tau}-\psi_{\tau}\right\|_{\C^{0}(\Omega;\R^{n})},$
  we obtain:
  \begin{align*}
    \left\| \mathcal{F}(\lambda,\phi)_{t} - \mathcal{F}(\lambda,\psi)_{t}\right\|_{\C^{0}(\Omega;\R^{n})}
    &\leq 2L_{r}\int_{0}^{t}\left\|\phi_{\tau}-\psi_{\tau}\right\|_{\C^{0}(\Omega;\R^{n})}\d\tau.
  \end{align*}
  Then, for all \( t\in I \),
  \begin{align*}
    e^{-\sigma t}\left\| \mathcal{F}(\lambda,\phi)_{t} - \mathcal{F}(\lambda,\psi)_{t}\right\|_{\C^{0}(\Omega;\R^{n})}
    &\leq 2L_{r}e^{-\sigma t}\int_{0}^{t}e^{\sigma \tau}\left\|\phi-\psi\right\|_{\sigma}\d\tau\\
    &\leq \frac{2L_{r}}{\sigma}\left\|\phi-\psi\right\|_{\sigma},
  \end{align*}
  which means that \( \mathcal{F}(\lambda,\cdot) \) is contractive for any \( \sigma>2L_{r} \).
\end{proof}

\begin{proofof}{Theorem~\ref{thm:dflow}}
Let $\sigma$ be chosen so that $\sigma>2L_r$ as in the proof of Lemma~\ref{lem:contraction}.
  By definition, \( X^{\lambda} \) is a fixed point of \( \mathcal{F}(\lambda,\cdot) \) for any \( \lambda\in [0,1] \).
  Therefore, by Theorem A.2.1 in~\cite{BressanPiccoli2007},
  \begin{equation}
    \label{eq:key_derivative}
    \left\| X^{\lambda} - X - \lambda w \right\|_{\sigma}\leq \frac{1}{(1-\kappa)}\left\|\mathcal{F}\left(\lambda,X + \lambda w \right) -X - \lambda w  \right\|_{\sigma},
  \end{equation}
  where \( \kappa\doteq 2L_{r}/\sigma < 1 \).

  It remains to estimate the right-hand side of~\eqref{eq:key_derivative}.
  Since \( X = \mathcal{F}(0,X) \), we obtain
  \begin{align*}
    \mathcal{F}&\left(\lambda,X_{t} + \lambda w_{t} \right)(t,x) - X_{t}(x) \\
    &= \int_{0}^{t} \left[V_{\tau}\left(X_{\tau}(x)+\lambda w_{\tau}(x),\left(X_{\tau}+\lambda w_{\tau}\right)_{\sharp}\vartheta\right)-V_{\tau}\left(X_{\tau}(x),X_{\tau\sharp}\vartheta\right)\right]\d\tau\\
    &+ \lambda\int_{0}^{t} W_{\tau}\left(X_{\tau}(x)+\lambda w_{\tau}(x),\left(X_{\tau}+\lambda w_{\tau}\right)_{\sharp}\vartheta\right)\d\tau.
  \end{align*}
Lemma~\ref{lem:O2-mod} demonstrates that the first integrand is equal to
  \begin{align*}
    \lambda D_{x}V_{\tau}\left(X_{\tau}(x),X_{\tau\sharp}\vartheta\right)w_{\tau}(x)
    &+ \lambda\int D_{\mu}V_{\tau}\left(X_{\tau}(x),X_{\tau\sharp}\vartheta,X_{\tau}(y)\right)w_{\tau}\left(y\right)\d \vartheta(y)\\
    &+\Oloc(t,x,\vartheta;\lambda^{2}),
  \end{align*}
  and the second one can be rewritten as
 \(  \lambda W_{\tau} \left( X_{\tau}(x), X_{\tau\sharp}\vartheta \right) + \Oloc(t,x,\vartheta;\lambda^{2})\).
  Now, the statement follows from~\eqref{eq:w}.
  The fact that \( C_{\Omega} \) depends only on \( \rho \), \( C_{\rho} \), \( r \), \( L_{r} \) is the consequence of~\eqref{eq:Rr}, \eqref{eq:lip_bound} and the second part of Lemma~\ref{lem:O2-mod}.
\end{proofof}

\section{Increment formula}
\label{sec:incr}

Now, we turn to the analysis of the increment of the cost functional along an adequate class of control variations. The theory of Pontryagin's maximum principle is commonly built around the class of needle-shaped variation. However, for the specified control-affine case, the latter can be replaced by a simpler class of \emph{weak} control variations. 

\subsection{Problem specification}

In this section, in order to simplify the presentation, we assume that the driving vector field $V$ is affine in control variable $u$, i.e.,
\begin{equation}
    V_{t}(x, \mu, u) = V^0_{t}(x, \mu) + \sum_{j=1}^{m} V^j_{t}(x, \mu) \, u_j,\quad{u_j\in \mathbb R},\label{eq:VF-specif}
\end{equation}
and the running cost is identically zero, i.e., $L \equiv 0$.
Later, in Sect.~\ref{sec:gen} we will discuss how to deal with the general case.
We begin by listing our basic assumptions.
\vspace{5pt}

\noindent\textbf{Assumption $(\bm A_{1})$:}
\begin{enumerate}
  \item \( V \) takes the form~\eqref{eq:VF-specif}, where all \( V^{j} \) with \( 0\leq j\leq m \) are of class \( \C^{1,1} \);
  \item \( U\subset \R^{m} \) is compact and convex;
  \item \( \ell\colon \P_{c}(\R^{n})\to \R \) is of class \( \C^{1,1} \).
\end{enumerate}

\vspace{5pt}

\noindent\textbf{Assumption \( (\bm A_{2}) \):} all maps \( D_{x}V^{j} \) with \( 0\leq j\leq m \) are continuously differentiable in \( x \) and their derivatives are locally bounded.

\begin{remark}
  \label{rem:smooth_w}
  We use Assumption \( (\bm A_{2}) \) only once: to show that, for any fixed control function \( u\in \mathcal{U} \), the solution \( w \) of~\eqref{eq:w}, \eqref{eq:w_init} which corresponds to \( V_{t}(x,\mu)\doteq V_{t}(x,\mu,u(t)) \) is \( \C^{1} \) in \( x \).
  Indeed, \( t\mapsto w^{\vartheta}_{t}(x) \) satisfies the ODE: $\displaystyle\frac{d}{dt}w_{t} = A(t,x)w_{t} + b(t,x),$
  where both functions
  \begin{align*}
    A(t,x) &\doteq D_xV_{t}\left(X^{\vartheta}_{t}(x),X^{\vartheta}_{t\sharp}\vartheta\right),\mbox{ and}\\
    b(t,x) &\doteq \int D_{\mu}V_{t}\left(X_{t}^{\vartheta}(x),X^{\vartheta}_{t\sharp}\vartheta, X^{\vartheta}_{t}(y)\right)w^{\vartheta}_{t}\left(y\right)\d\vartheta(y)
    + W_{t}\left(X^{\vartheta}_{t}(x) ,X^{\vartheta}_{t\sharp}\vartheta\right)
 \end{align*}
  are continuously differentiable. Hence, \( w \) is \( \mathcal C^1 \) in \( x \), according to the standard ODE theory.
\end{remark}

\subsection{Increment formula I}

Further in this section, \( \vartheta \) is supposed to be fixed, so we will omit it when writing the arguments \( X \) and \( w \).

Let us fix a pair of control functions \( u,\bar u\in \mathcal{U} \), $u \neq \bar u$. We call \( u \) a \emph{reference control} and \( \bar u \) a \emph{target control}.
A \emph{weak variation} of \( u \) towards \( \bar u \) is the convex combination
\begin{equation}
u^{\lambda} \doteq u + \lambda(\bar u - u), \quad \lambda \in [0,1].\label{eq:u-lambda}
\end{equation}

In view of \eqref{eq:VF-specif}, the variation \eqref{eq:u-lambda} implies the following perturbation of the reference vector field \( V_{t}(x,\mu) \doteq V_{t}\left(x,\mu,u(t)\right) \):
\begin{align*}
V^\lambda_t(x, \mu) \doteq &V_{t}\left(x, \mu,u^{\lambda}(t)\right) = V_{t}(x, \mu) + \lambda W_t(x, \mu),\\
W_t(x, \mu) \doteq & \sum_{j=1}^{m} V^j_{t}(x, \mu) \, \left(\bar u_j(t) - u_j(t)\right).
\end{align*}

Note that, by Assumption \( (\bm A_{1}) \), there exists \( C>0 \) such that \( \left\vert V^{\lambda}_{t}(x,\mu) \right\vert\leq C(1+\vert x\vert) \), for all \( t \), \( x \), \( \mu \), \( \lambda \), \( u \), \( \bar u \).
This means that \( \rho \) from~\eqref{eq:rho} can be chosen independently from \( u, \bar u\in \mathcal{U} \).
Again, by Assumption \( (\bm A_{1}) \), we can find \( C_{\rho} \) which guarantees, for all \( u,\bar u\in \mathcal{U} \), the estimate~\eqref{eq:Cbound}, then construct \( r \) by~\eqref{eq:Rr} and find \( L_{r} \) such that~\eqref{eq:lip_bound} holds for all \( u,\bar u\in \mathcal{U} \).
Now, Theorem~\ref{thm:dflow} implies that
\begin{displaymath}
  X^{\lambda}_{T} - X_{T} - \lambda w_{T} = \Oloc(x,\vartheta,u,\bar u;\lambda^{2}),
\end{displaymath}
{where $w$ is a solution of~\eqref{eq:w}, \eqref{eq:w_init}.}
Here, we think of \( \mathcal{U} \) as a compact topological space equipped with the weak-\( * \) topology \( \sigma(L^{\infty},L^{1}) \).

Since \( \mathcal{I}[u] =  \ell(X_{T\sharp}\vartheta) \) and \( \mathcal{I}[u^{\lambda}] =  \ell(X^{\lambda}_{T\sharp}\vartheta) \) and \( \vartheta \) is fixed, we can use Lemma~\ref{lem:O2-mod2} to get
{
\begin{proposition}
Under assumptions \( (\bm A_{1}) \), \( (\bm A_{2}) \) one has
\begin{equation}
  \label{eq:increment_draft}
  \mathcal{I}[u^{\lambda}] - \mathcal{I}[u] = \lambda\int D_{\mu}\ell \left(X_{T\sharp}\vartheta,X_{T}(y)\right)\, w_T\left(y\right)\d\vartheta(y) + \mathcal{O}(u,\bar u;\lambda^{2}),
\end{equation}
where $w$ is a solution of the linear problem \eqref{eq:w}, \eqref{eq:w_init}.
\end{proposition}
}
Here we write \( \mathcal{O} \) instead of \( \Oloc \) because \( \mathcal{U} \) is already compact.

Our next goal is to rewrite this formula in a ``constructive'' form, namely, in terms of a Hamiltonian system associated to our optimal control problem.

\subsection{Hamiltonian system}

The Hamiltonian system associated with Problem $(P)$ (see~\eqref{eq:cost}-\eqref{eq:control}) is merely a continuity equation on the cotangent bundle of \( \R^n \), i.e., on the space $\R^{n}\times (\R^n)^* \simeq \R^{2n}$ comprised by pairs $(x, p)$, where \( x \) is the primal and \( p \) is the dual state variables.
In our case, this equation takes the form
\begin{equation}
      \label{eq:hamsys}
      \partial_t \gamma_t + \div_{(x,p)}\left(\vec{H}\left(\cdot, \cdot,\gamma_t, u(t)\right)\gamma_t\right)=0,
\end{equation}
\begin{equation}
\label{eq:DH}
  \vec{H}(x, p, \gamma, u)\doteq
  \begin{pmatrix}
    \displaystyle V_{t}(x, \pi^1_\sharp\gamma, u)\\[0.2cm]
    \displaystyle- p\, D_x V_t(x, \pi^1_\sharp\gamma,u) - \iint  q \, D_\mu V_t(y, \pi^1_\sharp\gamma, u, x)\d \gamma(y,q)
  \end{pmatrix}.
\end{equation}
This equation is supplemented with the terminal condition
\begin{equation}
  \label{eq:terminal}
  \gamma_{T} = \left(\id,-D_{\mu}\ell (\mu_T)\right)_{\sharp}\mu_{T},
\end{equation}
where \( \mu_{t} \) satisfies~\eqref{eq:conteq}. The standard well-posedness result for nonlocal continuity equations (see, e.g., \cite{POGODAEV20203585}) guarantees that~\eqref{eq:hamsys}, \eqref{eq:terminal} has a unique solution \( \gamma_{t} \).
Moreover, the projection of \( \gamma_{t} \) onto the \( x \) space coincides with \( \mu_{t} \):
\begin{equation}
\pi^{1}_{\sharp}\gamma_{t} = \mu_{t}\quad \forall t\in I.\label{eq:gamma-mu}
\end{equation}

\subsection{Increment formula II}

Let us go back to~\eqref{eq:increment_draft}. First, recalling that $\mu_T\doteq X_{T\sharp}\vartheta$, we express the integral entering in its right-hand side as follows:
\begin{align*}
  \int D_{\mu}\ell (\mu_T, x)\, & w_T\left(X_{T}^{-1}(x)\right)\d\mu_{T}(x)\\
  &= -\iint p \, w_T\left(X_{T}^{-1}(x)\right)\d\left[\left(\id,-D_{\mu}\ell (\mu_T)\right)_{\sharp}\mu_{T}\right](x,p)\\
  &= -\iint p \, w_T\left(X_{T}^{-1}(x)\right)\d\gamma_{T}(x,p).
\end{align*}
By Lemma~8.1.2~\cite{AGS}, the following version of the classical Newton-Leibniz formula holds for any function $\psi \in \C^1(I\times \R^{2n})$:
\begin{align}
&\iint \psi_T \d\gamma_{T} - \iint \psi_0 \d\gamma_{0} = \int_{0}^{T}\Big(\iint \Xi_t(x, p) \d \gamma_{t}(x,p)\Big)\d t,\label{eq:hamsys-weak-1}\\
  \Xi_t(x, p)
  &\doteq \partial_{t}\psi_t(x,p) + \nabla_{x}\psi_t(x,p) \, V_t\left(x, \mu_t\right) \nonumber \\
  &-\Big[p\, D_x V_t\left(x, \mu_t\right) + \iint  q \, D_\mu V_t\left(y, \mu_t,x \right)\d \gamma_t(y,q)\Big] \nabla_{p}\psi_t(x,p).\label{eq:hamsys-weak-2}
\end{align}
Remark~\ref{rem:smooth_w} allows us to take \( \psi_t(x,p) \doteq p \cdot w_t\left(X_{t}^{-1}(x)\right) \) in the above expression.
Recall that \( X_{t}(x) = P_{0,t}(x) \), where \( P \) is the flow of the noauthonomous vector field \( v_{t}(x) = V_{t}\left(x,\mu_{t},u(t)\right) \), in particular, \( X_{t}^{-1} = P_{t,0} \) and we can use the standard rules of flow differentiation (Theorem 2.3.3~\cite{BressanPiccoli2007}) to perform the calculations:
{
\begin{align*}
  \partial_{t}\psi_t(x,p)
  &= p \left[\partial_{t}w_t \left(P_{t,0}(x)\right) + D_x w_t\left(P_{t,0}(x)\right) \partial_t P_{t,0}(x)\right]\\
  &= p \left[\partial_{t}w_t\left(P_{t,0}(x)\right) - D_x w_t \left(P_{t,0}(x)\right) \, D_{x}P_{t,0}(x) \,  V_t\left(x, \mu_t\right)  \right]\\
  &= p \, \partial_{t}w_t\left(P_{t,0}(x)\right) - \nabla_{x}\psi_t(x,p) \cdot V_t\left(x, \mu_t\right),\\
  \nabla_x \psi_t(x,p)
   & = p \, D_x w_t\left(P_{t,0}(x)\right)  \, D_xP_{t,0}(x),\quad
  \nabla_p \psi_t(x,p)
  = w_t\left(P_{t,0}(x)\right).
\end{align*}
}
Then,
\begin{align*}
\iint \Xi_t \d \gamma_{t} = & \iint \left[p \, \partial_{t}w_t\left(X_{t}^{-1}(x)\right) - p \, D_x V_t\left(x, \mu_t\right) w_t\left(X_{t}^{-1}(x)\right)\right] \d \gamma_t(x, p)\\
- &\iint \Big[\iint  qD_\mu V_t\left(y, \mu_t,x\right)\d \gamma_t(y,q)\Big]  \, w_t\left(X_{t}^{-1}(x)\right)\d \gamma_t(x, p).
\end{align*}
In view of \eqref{eq:w}, the right-hand side
reduces to
\begin{align*}
\iint p \, W_t\left(x,\mu_{t}\right)\d \gamma_t(x, p)
  &+\iint p \Big[\int D_{\mu}V_t(x,\mu_{t},y)\, w_{t}\left(X_{t}^{-1}(y)\right)\d\mu_{t}(y)\Big]\d \gamma_t(x, p)\\
  &-  \iint \Big[\iint  q \, D_\mu V_t\left(y, \mu_t,x\right)\d \gamma_t(y,q)\Big]  \, w_t\left(X_{t}^{-1}(x)\right)\d \gamma_t(x, p).
\end{align*}
Renaming the variables $(x, p) \leftrightarrow (y, q)$ in the latter term shows that the last two terms cancel out.
Hence,
\[
\iint \Xi_t \d \gamma_{t} = \iint p \, W_t\left(x,\mu_{t}\right)\d \gamma_t(x, p).
\]
Finally, noticing that \( \psi_0 \doteq p \, w_0 \equiv 0 \) and
\[
  \lambda \iint\psi_T \d\gamma_{T} = \lambda \iint p \, w_T\left(X_{T}^{-1}(x)\right)\d\gamma_{T}(x,p) =
- \left(\mathcal{I}[u^{\lambda}] - \mathcal{I}[u] - \mathcal{O}(u,\bar u;\lambda^{2})\right),
\]
then using~\eqref{eq:hamsys-weak-1} and the definition of $W_t$, we have
{
\begin{proposition}
Under assumptions \( (\bm A_{1}) \), \( (\bm A_{2}) \) it holds
\begin{align}
  \mathcal{I}[u^\lambda] - \mathcal{I}[u]
  = -\lambda\int_0^T\left(H_{t}\left(\gamma_t, \bar u(t)\right) - H_{t}\left(\gamma_t, u(t)\right)\right)\d t+\mathcal{O}(u,\bar u;\lambda^{2}),
    \label{eq:incr}
\end{align}
where
\begin{equation}
  \label{eq:hamilt}
  H_{t}(\gamma, u) \doteq  \iint p \, V_{t}(x, \pi^1_\sharp \gamma,u)\d{\gamma(x,p)}
\end{equation}
and $t \mapsto \gamma_t$ is a solution of the Hamiltonian system~\eqref{eq:hamsys}--\eqref{eq:terminal}.
\end{proposition}
}
\subsection{Pontryagin's maximum principle}

A consequence of the increment formula~\eqref{eq:incr} is the following version of Pontryagin's maximum principle.

\begin{theorem}[PMP in terms of Hamiltonian system]
  \label{thm:pmp}
  Assume that \((\bm A_{1,2})\) hold, and \( \vartheta\in \mathcal P_{c}(\mathbb{R}^{n}) \).
  Let $(\mu,u)$ be an optimal pair for $(P)$.
  Then \( u(t) \) satisfies, for a.e. $t\in I$, the maximum condition
    \begin{equation}
    H_{t}\left(\gamma_t, u(t)\right) = \max_{\upsilon \in U} H_{t}(\gamma_t, \upsilon),\label{eq:maximum-cond-limit}
    \end{equation}
    where \(\gamma\) is a unique solution of the Hamiltonian system \eqref{eq:hamsys}--\eqref{eq:terminal} and \( H_{t} \) is defined by~\eqref{eq:hamilt}.
\end{theorem}
\begin{proof}
  Since \( u \) is optimal, we have \( \mathcal{I}[u^{\lambda}] - \mathcal{I}[u] \geq 0 \) for any target control \( \bar u \).
  Now, the increment formula implies that
  \begin{displaymath}
    \sup_{\bar u \in \mathcal{U}}\int_0^TH_{t}\left(\gamma_t, \bar u(t)\right)\d t = \int_{0}^{T} H_{t}\left(\gamma_t, u(t)\right)\d t.
  \end{displaymath}
  On the other hand,
  \begin{equation}
    \label{eq:HH}
    \sup_{\bar u \in \mathcal{U}}\int_0^TH_{t}\left(\gamma_t, \bar u(t)\right)\d t
    \leq \int_{0}^{T}\max_{\upsilon\in U}H_{t}\left(\gamma_t, \upsilon\right)\d t.
  \end{equation}
  Let \( \psi(t,\upsilon) \doteq H_{t}(\gamma_{t},\upsilon) \) and \( \alpha(t) \doteq \max_{\upsilon\in U}\psi(t,\upsilon) \).
  It is easy to check that \( \psi \) is a Carath\'eodory map.
  Since \( \alpha(t)\in \psi(t,U) \) for a.e. \( t\in I \), we deduce from {Filippov's lemma~\cite[Theorem 8.2.10]{aubinSetvaluedAnalysis2009}} that there exists \( \tilde u\in \mathcal{U} \) satisfying \( \alpha(t) = \psi(t,\tilde u(t)) \) for a.e. \( t\in I \).
  Hence, the inequality in~\eqref{eq:HH} can be replaced by the equality, which completes the proof.
\end{proof}
{
\begin{remark}\label{rem:comparison}
Pontryagin's maximum principle displayed by Theorem~\ref{thm:pmp} is essentially 
the same as in \cite{BonnetFrankowska2021a,Bonnet-Rossi-2021}. However, in these papers, the driving vector field has a specific form: It can be represented as the sum of a nonlocal drift term and an external Lipschitz vector field $u=u(t,x)$ playing the role of control action. In our case, the control is a measurable function of time variable only $u=u(t)$, which may enter in the {non-local term} itself, thus enabling us, e.g., to govern convolution kernels as in \eqref{VF-conv}. Finally, note that Theorem~\ref{thm:pmp} can be derived from the (most general) version of PMP recently obtained in  \cite{averboukh2022pontryagin}, which relies on the so-called Lagrangian interpretation \cite{CAVAGNARI2022268} of the mean-field control problem $(P)$. 
\end{remark}
}
\begin{remark}\label{rem:num-pitfall}
We conclude this section by stressing two obvious drawbacks of the presented form of the necessary optimality condition, which are critical for its numerical implementation.
\begin{enumerate}
    \item Equation \eqref{eq:hamsys} is defined on the space of dimension $2n$, which makes its numerical solution computationally demanding even for $n=2$.

  \item Even if $\mu$ is absolutely continuous, $\gamma$ is not. In other words, $\gamma$ never takes the form $\varrho_t \, \mathcal L^{2n}$ with a density function $(t,x,p) \mapsto \varrho_t(x,p)$. This is due to the fact that $\gamma_T$ is supported on the graph of the map $x \mapsto -D_{\mu}\ell (\mu_T)(x)$, which is always $\mathcal L^{2n}$-null set. This means that system \eqref{eq:hamsys} can not be solved by standard numerical methods for hyperbolic PDEs, which can be used only when densities exist.
\end{enumerate}
These issues motivate the development of a new version of Theorem~\ref{thm:pmp}, which is obtained by extracting the ``adjoint system'' from the Hamiltonian PDE~\eqref{eq:hamsys}. 
\end{remark}

\section{Adjoint equation}
\label{sec:adjoint}

It this section, we shall see that
the Hamiltonian system~\eqref{eq:hamsys} can be decoupled into the primal and dual parts just as one is used to experience in the classical optimal control theory. This fact will allow us to rewrite the increment formula and Pontryagin's maximum principle in an equivalent form, suitable for numerics. 

\subsection{Derivation}

After reflecting upon the formula  \eqref{eq:incr}, one comes up with an idea to take, as a matter of adjoint trajectory, the family of signed vector (namely, row vector) measures defined by
\begin{equation}
  \label{eq:nu}
  \langle \nu_{t},\phi\rangle \doteq \int p\,\phi(x) \d\gamma_{t}(x,p), \quad \phi \in \C^1(\R^n;\R^{n}),
\end{equation}
where \( \gamma_{t} \) is the solution of~\eqref{eq:hamsys}--\eqref{eq:terminal}.
Indeed, return to representation~\eqref{eq:hamsys-weak-1}, \eqref{eq:hamsys-weak-2} and specify the class of test functions $\psi$ as follows:
\[
\psi_t(x, p)= p\,\phi_t(x), \quad  \phi\in \C^1_{c}\left((0,T)\times\R^{n};\R^{n}\right).
\]
In this case, the left-hand side of \eqref{eq:hamsys-weak-1} vanishes, which implies
\begin{equation}
\int_{0}^{T}\Big(\iint \Xi_t \d \gamma_{t}\Big)\d t = 0,\label{h-w}
\end{equation}
where  $\Xi$ is defined in \eqref{eq:hamsys-weak-2}. 
In terms of $\nu$, the parts of the integral in the left hand side of \eqref{h-w} can be represented as follows:
\begin{align*}
 \iint \partial_{t}\psi_t(x,p) \d\gamma_{t}(x,p) &= \iint p\,\partial_{t} \phi_t(x)\d \gamma_{t}(x,p) =
  \int \partial_{t}\phi_t(x)\cdot\d\nu_{t}(x),\\
  \iint \nabla_{x}\psi_t(x,p) \, V_t(x, \mu_t)\d \gamma_{t}(x,p)&=
  \iint p \,D_x\phi_t(x) \, V_t(x, \mu_t)\d \gamma_{t}(x,p)\\
  &=\int D_x\phi_t(x) \, V_t(x, \mu_t)\cdot  \d\nu_{t}(x),\\
  \iint p\,D_{x}V_{t}(x,\mu_{t})\,\nabla_{p}\psi_t(x,p)\d\gamma_{t}(x,p) & = \int D_{x}V_{t}(x,\mu_{t})\, \phi_t(x)\cdot \d \nu_{t}(x),
\end{align*}
and, according to~\eqref{eq:gamma-mu},
\begin{align*}
  &\iint\bigg(\iint  q D_\mu V_t\left(y, \mu_t,x\right)\d \gamma_t(y,q)\bigg)\nabla_p\psi_t(x,p)\d \gamma_t(x,p)\\
  &= \iint\bigg(\int D_\mu V_t\left(y, \mu_t,x\right)\phi_t(x)\cdot \d \nu_t(y)\bigg)\d \gamma_t(x,p)\\
  &=
  \int\bigg(\int D_\mu V_t\left(y, \mu_t,x\right)\phi_{t}(x)\cdot\d \nu_t(y)\bigg)\d \mu_t(x).
\end{align*}
Substituting these expressions into \eqref{h-w}, we obtain
\begin{align}
0 & \equiv \int_0^T \Big[\int\left(\partial_{t}\phi_t(x) + D_x\phi_t(x) \, V_t(x, \mu_t)\right)\cdot \d \nu_t(x)\Big] \d t\,\nonumber\\
  &-\int_0^T \!\int  D_x V_t\left(x, \mu_t\right)\,\phi_{t}(x)\cdot \d \nu_t(x)\d t\notag\\ 
  &- \int_0^T\int\!\Big(\!\int  D_\mu V_t\left(y, \mu_t,x\right)\phi_{t}(x)\cdot \d \nu_t(y)\Big)\d \mu_t(x)\d t.\label{eq:adjoint-weak}
\end{align}
The choice \( \phi(x) = (0,\ldots,\phi^{i}(x),\ldots 0)^{{T}} \), where only \( i \)-th component of \( \phi \) is nonzero, shows that this is merely the weak formulation of the following system of balance laws:
\begin{align}
  \label{eq:adjoint}
\partial_{t}\nu_{i} + \div_x(v \, \nu_i) &= \sum_{j} \left[\Big(\int m^{j}_{i} \d \nu_j\Big)\mu - \partial_{x_{i}}v^{j} \, \nu_{j}\right], \quad 1\leq i\leq n.
\end{align}
Here, for the sake of readability, we omit the lower index \( t \) of $\nu_t$ and abbreviate
\begin{displaymath}
 v_{t}(x)\doteq V_{t}\left(x,\mu_{t},u(t)\right),\qquad
 \int m^{j}_{i}\d\nu_{j} \doteq \int m^{j}_{i}(t,x,y)\d(\nu_{j})_{t}(y),
\end{displaymath}
where \( m^{j}_{i} = m^{j}_{i}(t,x,y) \) are elements of the matrix \( D_{\mu}V_{t}\left(y,\mu_{t},u(t),x\right) \).

At the final time instant $T$, one has
\[
  \langle \nu_{T},\phi\rangle = \iint p\,\phi(x)\d\left[(\id,-D_\mu\ell(\mu_T))_{\sharp}\mu_{T}\right](x,p) = -\int D_\mu\ell(\mu_T)(x)\,\phi(x)\d\mu_{T}(x),
\]
which can be rewritten in terms of the Radon-Nikodym derivative as
\begin{equation}
  \label{eq:adj_term}
  \frac{d\nu_{T}}{d\mu_{T}} = -D_{\mu}\ell(\mu_{T}).
\end{equation}

\begin{definition}
  \label{def:adjoint}
  We call the backward system~\eqref{eq:adjoint},~\eqref{eq:adj_term} of nonlocal linear PDEs the \emph{adjoint system} associated to the optimal control problem \( (P) \).
\end{definition}

\subsection{Well-posedness}

We observe that there exists a solution of the adjoint system, namely, the one defined by~\eqref{eq:nu}. Let us show that this solution is unique. Basically, the adjoint system \eqref{eq:adjoint} is a system of linear balance laws 
with sources of the form
\begin{equation}
  \label{eq:source}
\varsigma_{i} \doteq   \sum_{j=1}^{n} \Big[\Big(\int m^{j}_{i}\d\nu_{j}\Big) \mu-\partial_{x_{i}}v^{j}\nu_{j} \Big].
\end{equation}

To proceed, recall basic properties~\cite{PogStar2022} of the linear balance law
\begin{equation}
\label{eq:linear}
\partial_t \varrho_t + \div_{x}(f_t\varrho_t) = \varsigma_t
\end{equation}
with a Carath\'eodory, locally Lipschitz, sublinear vector field \( f_{t} \) and an integrable source \( \varsigma_{t} \).

\begin{definition}
  \label{def:integrable}
A curve $\varsigma\colon I\to \mathcal{M}(\R^n)$ is called \emph{integrable} if for any Borel set \( A\subset \R^{d} \) the map \(t\mapsto\varsigma_{t}(A)\) is measurable and \( \int_{0}^{T}\|\varsigma_{t}\|_{TV}\,d t < +\infty \), where \( \|\cdot\|_{TV} \) denotes the total variation norm on \( \mathcal{M}(\R^{n}) \).
\end{definition}

For integrable curves we can define a notion of \emph{integral} in the usual way:
$\displaystyle\left(\int_{0}^{t}\varsigma_{s}\d s\right)(A) \doteq  \int_{0}^{t}\varsigma_{s}(A)\d s$,
for all Borel sets \( A\subset \R^{n} \).

\begin{definition}\label{def-sol}
 A curve \( \varrho\in C\big(I;\mathcal M(\R^{n})\big) \) is called a solution of~\eqref{eq:linear} if and only if, for any test function \( \varphi\in C_{c}^{\infty}(\R^{n})\) and a.e. \( t\in I \), one has
\begin{displaymath}
 \frac{d}{dt}\int\varphi\,d \varrho_{t} = \int \nabla\varphi \cdot f_{t}\,d \varrho_{t} + \int \varphi\,d \varsigma_{t}.
\end{displaymath}
\end{definition}

\begin{theorem}
  \label{thm:balance}
  Under our assumptions, there exists a unique solution of~\eqref{eq:linear} with the initial condition $\varrho_0=\xi$. Moreover, it can be expressed by 
    \begin{equation}
    \label{eq:linsol}
    \varrho_t = P_{0,t\sharp}\xi + \int_0^t P_{s,t\sharp}\varsigma_s\,d s,\quad t\in I,
    \end{equation}
  where \( P \) is the flow of \( f_{t} \).
\end{theorem}

The following Lemma collects several well-known properties of the total variation norm (since their proof is quite standard, we drop them for brevity). 
\begin{lemma}
  \label{lem:tv}
  Let \( \varrho\in \mathcal{M}(\R^{n}) \) and \( \varsigma\colon I\to \mathcal{M}(\R^{n}) \) be an integrable curve.
  Then,
  \begin{enumerate}
    \item for any Borel measurable bijective map \( f\colon\R^{n}\to\R^{n} \),
    \begin{displaymath}
      \|f_{\sharp}\varrho\|_{TV} = \|\varrho\|_{TV};
    \end{displaymath}
    \item for all \( t\in I \),
    \begin{displaymath}
    \Big\| \int_{0}^{t}\varsigma_{s}\d s \Big\|_{TV}\leq \int_{0}^{t}\|\varsigma_{s}\|_{TV}\d s;
    \end{displaymath}
    \item if \( \spt \varrho \) is contained in a compact set \( \Omega \), then for any \( \phi \in \C^{0}(\R^{n}) \)
    \begin{displaymath}
    \|\phi\varrho\|_{TV} \leq \|\phi\|_{\C^{0}(\Omega)}\,\|\varrho\|_{TV};
    \end{displaymath}
    \item if \( \spt \varrho \) is contained in a compact set \( \Omega \), then for any \( K \in \C^{0}(\R^{2n}) \)
          \begin{displaymath}
          \left\vert\int K(x,y)\d\varrho(y)\right\vert \leq \|K\|_{\C^{0}(\Omega^{2})}\,\|\varrho\|_{TV}\quad \forall x\in \Omega.
          \end{displaymath}
  \end{enumerate}
\end{lemma}
The well-posedness of the adjoint system is established by the following result, where $\mathcal{M}_{c}(\R^{n})$ denotes the subset of $\mathcal{M}(\R^{n})$ composed of signed measures with compact support.
\begin{proposition}
  \label{prop:unique}
  Under assumptions \((\bm A_{1,2})\), the adjoint system~\eqref{eq:adjoint} with the terminal condition \( \nu_{T} = \xi \), \( \xi\in \left[\mathcal{M}_{c}(\R^{n})\right]^n\), 
  has a unique solution.  
\end{proposition}
\begin{proof}
Take two terminal measures \( \xi, \xi' \in \left[\mathcal{M}_{c}(\R^{n})\right]^n \) and denote by \( \nu_{t} \) and \( \nu_{t}' \) the corresponding (potentially, non-unique) trajectories of~\eqref{eq:adjoint}.
Then, from Theorem~\ref{thm:balance} and Lemma~\ref{lem:tv}, it follows that
\begin{displaymath}
  \left\|(\nu_{i}-\nu_{i}')_t\right\|_{TV}\leq
  \left\|\xi_{i}-\xi_{i}'\right\|_{TV} +
  \int_0^t \left\|(\varsigma_i-\varsigma_{i}')_{s}\right\|_{TV}\,d s,
\end{displaymath}
where \( \varsigma \) and \( \varsigma' \) are the corresponding sources defined by~\eqref{eq:source}.
Since
\begin{displaymath}
  \varsigma_i-\varsigma_{i}' = \sum_{j=1}^{n} \partial_{x_{i}}v^{j}(\nu_{j} - \nu_{j}') - \sum_{j=1}^{n}\Big(\int m^{j}_{i}\d (\nu_{j} - \nu_{j}')\Big) \mu,
\end{displaymath}
we obtain, again by Lemma~\ref{lem:tv},
\begin{displaymath}
  \left\|\varsigma_i-\varsigma_{i}'\right\|_{TV} \leq C^{1}_{\Omega}\sum_{j=1}^{n} \left\|\nu_{j} - \nu_{j}'\right\|_{TV} + C^{2}_{\Omega}\|\mu\|_{TV} \sum_{j=1}^{n}\left\|\nu_{j} - \nu_{j}'\right\|_{TV},
\end{displaymath}
where \( \Omega \) is a compact set containing the supports of the measures \( \nu_{t} \), \( \nu_{t}' \), \( \mu_{t} \), \( t\in I \) (one can show that there is such a set by reasoning as in~\cite{PogStar2022}), \( C^{1}_{\Omega} \) is an upper bound of \( \sum_{i,j}\vert\partial_{x_{i}}v^{j}\vert \) on \( I\times \Omega \) and \( C^{2}_{\Omega} \) is an upper bound of \( \sum_{i,j}\vert m^{j}_{i}\vert \) on \( I\times\Omega\times\Omega \).

By letting \( r(t) = \displaystyle\sum_{i=1}^{n}\|( \nu_{i}-\nu_{i}')_{t}\|_{TV} \), we obtain
\begin{displaymath}
  r(t) \leq \sum_{i=1}^{n}\left\|\xi_{i}-\xi_{i}'\right\|_{TV} + n\int_{0}^{t}\left(C^{1}_{\Omega} + C^{2}_{\Omega}\|\mu\|_{TV}\right)r(s)\d s.
\end{displaymath}
Now, Gr\"onwall's lemma gives the uniqueness.
\end{proof}

\subsection{Increment formula III}

The increment formula~\eqref{eq:incr} and Pontryagin's maximum principle (Theorem~\ref{thm:pmp}) are trivially reformulated in terms of a solution to the adjoint system.

\begin{theorem}[Increment formula]
  \label{thm:increment}
  Assume that \((\bm A_{1,2})\) hold, and \( \vartheta\in \mathcal P_{c}(\mathbb{R}^{n}) \).
  Let \( u,\bar u\in \mathcal{U} \) and \( u^{\lambda} = u+\lambda(\bar u - u) \), \( \lambda\in [0,1] \), be the weak variation of \( u \).
  Then,
  \begin{equation}
  \mathcal{I}[u^\lambda] - \mathcal{I}[u] = -\lambda\int_0^T\left(\mathbf H_{t}\left(\mu_t,\nu_{t}, \bar u(t)\right) - \mathbf H_{t}\left(\mu_{t}, \nu_t, u(t)\right)\right)\d t + \mathcal{O}(u,\bar u;\lambda^{2}),\label{eq:incr2}
  \end{equation}
  where
  \begin{equation}
    \label{eq:hamilt2}
        \mathbf H_{t}(\mu, \nu, u) \doteq  \int V_{t}(x, \mu, u)\cdot \d\nu(x).
  \end{equation}
\end{theorem}

\begin{theorem}[PMP in terms of the adjoint system]
  \label{thm:pmp2}
  Assume that \((\bm A_{1,2})\) hold, and \( \vartheta\in \mathcal P_{c}(\mathbb{R}^{n}) \).
  Let $(\mu,u)$ be an optimal pair for $(P)$.
  Then \( u \) satisfies, for a.e. $t\in I$, the maximum condition
    \begin{equation}
    \mathbf H_{t}\left(\mu_t,\nu_{t}, u(t)\right) = \max_{\upsilon \in U} \mathbf H_{t}(\mu_t,\nu_{t}, \upsilon),\label{eq:maximum-cond-limit2}
    \end{equation}
    where \(\nu\) is a unique solution of the adjoint system~\eqref{eq:adjoint}, \eqref{eq:adj_term} and \( \textbf H_{t} \) is defined by~\eqref{eq:hamilt2}.
\end{theorem}

\begin{remark}
  \label{rem:ac}
  Since the adjoint system~\eqref{eq:adjoint}, \eqref{eq:adj_term} has a unique solution \( \nu_{t} \), it must coincide with the one given by~\eqref{eq:nu}.
  In particular, \( \nu_{t} \) acts on test functions \( \phi\in \C^{1}(\R^{n};\R^{n}) \) by the rule
  \begin{equation}
    \label{eq:ac}
  \langle\nu_{t},\phi\rangle=\int p \, \phi(x)\d\gamma_{t}(x,p) = \int\Big(\int p\d \gamma_{t}^{x}(p)\Big)\phi(x)\d\mu_{t}(x),
\end{equation}
where \( \gamma^{x}_{t} \) is the disintegration of \( \gamma_{t} \) with respect to \( \mu_{t} \) {(see~\cite[Theorem 5.3.1]{AGS})}.
If the initial measure \( \vartheta \) is absolutely continuous with respect to the Lebesgue measure \( \mathcal{L}^{n} \), then so are all \( \mu_{t} \), \( t\in I \) (thanks to the representation \( \mu_{t} = X_{t\sharp}\vartheta \)).
Now,~\eqref{eq:ac} implies that every \( \nu_{t} \), \( t\in I \), must be absolutely continuous as well.

The discussed fact has important consequences, which answer the challenges outlined by Remark~\ref{rem:num-pitfall}:
\begin{enumerate}
\item In contrast to the Hamiltonian continuity equation \eqref{eq:hamsys} as a whole, the adjoint system is solvable numerically.

\item While handling the adjoint equation, we deal with a system of $n+1$ first-order hyperbolic PDEs, each one ``living'' on $\R^n$. Solving this system is less computationally expensive than treating 
a single equation on $\R^{2n}$.
\end{enumerate}
\end{remark}

\subsection{Linear case}

Now, we establish a connection between Theorem~\ref{thm:pmp2} and the well-known version of PMP for \(\mu\)-independent vector fields $V_{t}(x,\mu,u)=V_{t}(x,u)$ (see, e.g.,~\cite{Pogodaev2019, Bonnet2021AMT}). For such fields, the part of adjoint state is played by a solution $(t, x)\mapsto \psi_t(x)$ of a single non-conservative transport equation
  \begin{equation}
  \partial_t \psi + \nabla_{x} \psi \, V_{t} =0.\label{eq:psi-loc}
  \end{equation}
It is reasonable to expect that, under sufficient regularity, the adjoint system \eqref{eq:adjoint} boils down to \eqref{eq:psi-loc}. This ansatz is confirmed by the following
\begin{proposition}\label{prop:mu-p-local}
  Assume that \((\mathbf{A_{1,2}})\) hold, \( x\mapsto\frac{\delta\ell}{\delta\mu}(\mu,x) \) is of class \( \mathcal{C}^{2} \), \( \vartheta\in \mathcal P_{c}(\mathbb{R}^{n}) \) and \( V_{t}(x,\mu,u) = V_{t}(x,u) \).
  Let $u \in \mathcal U$, and \( \mu_{t} \) and \( \nu_{t} \) be the corresponding solutions of~\eqref{eq:conteq} and~\eqref{eq:adjoint}, \eqref{eq:adj_term}, respectively. Then, for a.e. $t\in I$,
  \begin{equation}
      \label{eq:mu-p} \frac{\d \nu_{t}}{\d \mu_{t}} = \nabla_{x}\psi_t,
  \end{equation}
  where $(t, x) \mapsto \psi_t(x)$ is a solution of the transport equation \eqref{eq:psi-loc} with the terminal condition
  \begin{equation}
 \psi_T = - \frac{\delta\ell}{\delta\mu}(\mu_{T}).\label{eq:psi-loc-fin}
  \end{equation}
\end{proposition}
\begin{proof} 
  It is clear that the representation \eqref{eq:mu-p}, \eqref{eq:psi-loc-fin} does agree with the terminal condition \eqref{eq:adj_term}, since 
  \[
  -\nabla_x \psi_T \doteq D_{x}\frac{\delta \ell}{\delta\mu}(\mu_T) \doteq  D_{\mu}\ell(\mu_T).
  \]
  Due to the uniqueness of a solution to \eqref{eq:adjoint}, we only need to formally check that $\nu_i \doteq \partial_{x_{i}}  \psi \, \mu$, \( 1 \le i \le m \), meets the identity \eqref{eq:adjoint-weak} with the vector field $V_t(x, \mu_{t}, u(t)) = v_t(x)$.

  A solution of (\ref{eq:psi-loc}), (\ref{eq:psi-loc-fin}) can be written explicitly as \( \psi_{t}(x) = -\frac{\delta\ell}{\delta\mu}(\mu_{T}, P_{t,T}(x)) \), where \( P \) is the flow of \( v \).
  This formula, together with our assumptions, implies that \( \psi \) admits the partial derivatives \( \partial_{x_i}\psi \), \( \partial_{x_i}\partial_{x_j}\psi \) and \( \partial_{x_i}\partial_{t}\psi \) for all \( 1\le i,j \le m \).
  These derivatives are at least measurable in \( t \), continuous in \( x \) and locally bounded.
  Take the standard mollification kernel \( \eta_{\epsilon}\colon \mathbb{R}\to \mathbb{R} \) and consider the convolution
  \[
    \psi^{\epsilon}_{t}(x) = \int \eta_{\epsilon}(t-s)\psi_{s}(x)\d s.
  \]
  It is easy to see that \( \partial_{x_i}\partial_t\psi^{\epsilon} = \partial_t\partial_{x_i}\psi^{\epsilon} \) and \( \partial_{\alpha}\psi^{\epsilon} \to \partial_{\alpha}\psi \) as \( \epsilon\to 0 \) in the sense that
  \begin{equation}
    \label{eq:convergence}
    \int_{I} \sup_{x}  \left\vert \partial_{\alpha} \psi^{\epsilon} - \partial_{\alpha}\psi \right\vert\d t \to 0,
  \end{equation}
  where \( \partial_{\alpha} \) denotes any of the derivatives \( \partial_{x_i} \), \( \partial_{x_i}\partial_{x_j} \), \( \partial_{x_i}\partial_t \).
  Let \( \nu^{\epsilon}_i \doteq \partial_{x_i}\psi^{\epsilon}\mu \).
  Then, we can formally write
  \[
    \partial_t\nu_i^{\epsilon} = \partial_{t}(\partial_{x_i}\psi^{\epsilon}\mu) =\partial_{x_i}(\partial_t\psi^{\epsilon})\mu - \partial_{x_i} \psi^\epsilon \sum_j \partial_{x_j}(v^j\mu).
  \]
  More precisely, for any test function \( \phi \in C^\infty_c(I \times \mathbb{R}^n;\mathbb{R}) \), we have
  \begin{align*}
    -\left\langle \partial_{x_i}\psi^{\epsilon}\partial_{x_j}(v^j \mu),\phi \right\rangle
    &=  -\left\langle \partial_{x_j}(v^j\mu),\phi \, \partial_{x_i}\psi^{\epsilon}\right\rangle\\
    &=\left\langle v^j\mu, \partial_{x_j}\phi \, \partial_{x_i}\psi^{\epsilon} + \phi \, \partial_{x_i}\partial_{x_j}\psi^{\epsilon} \right\rangle\\
    &= \left\langle v^j\nu_i^{\epsilon},\partial_{x_j}\phi \right\rangle + \left\langle \mu, \phi \, \partial_{x_i}(v^j \, \partial_{x_j}\psi^{\epsilon}) \right\rangle - \left\langle \nu_j^{\epsilon}, \phi \ \partial_{x_i} v^j\right\rangle.
  \end{align*}
  Therefore,
  \begin{align*}
    \left\langle \partial_t\nu^{\epsilon}_i,\phi \right\rangle
    = \left\langle \mu, \partial_{x_i}\left(\partial_t\psi^{\epsilon} + \sum_{j}v^j \, \partial_{x_j}\psi^{\epsilon}\right) \, \phi \right\rangle
    &- \left\langle \sum_{j}\partial_{x_j}(v^j\nu^{\epsilon}_{i}),\phi \right\rangle- \left\langle \sum_{j}\partial_{x_i} v^j \nu^{\epsilon}_{j},\phi \right\rangle.
  \end{align*}
  It remains to use~(\ref{eq:convergence}) for passing to the limit as \( \epsilon\to 0 \).
  The first term in the right-hand side vanishes thanks to~(\ref{eq:psi-loc}), so we get
  \[
    \partial_t\nu_i + \sum_j\partial_{x_j}(v^j\nu_i) = - \sum_j\partial_{x_i}v^j\nu_j,
  \]
  in the sense of distributions.
  Since \( D_{\mu}V = 0 \), we conclude that \( \nu \) does satisfy (\ref{eq:adjoint}).
\end{proof}

\section{Descent method}
\label{sec:descent}

Now, we are able to construct an algorithm for the numerical solution of Problem $(P)$ with vector field as in~\eqref{eq:VF-specif}. {
Note that similar algorithms were earlier proposed for solving classical~\cite{arguchintsevOptimalControlNonlocal2009} and stochastic~\cite{annunziatoFokkerPlanckControl2013} optimal control problems.
}

\subsection{Algorithm}
\label{subsec:alg}
Let \( u \) be a reference control, \( \mu_{t} \) and \( \nu_{t} \) be the corresponding trajectory and co-trajectory.
Construct the target control as follows:
\begin{align*}
  \bar u(t) \doteq \arg\max_{\upsilon\in U} \mathbf H_{t}(\mu_{t},\nu_{t},\upsilon) = \arg\max_{\upsilon\in U}\sum_{j=1}^{m}\upsilon_{j}\int V^{j}_{t}(x,\mu_{t}) \cdot\d\nu_{t}(x).
\end{align*}
The increment formula~\eqref{eq:incr2} shows that \( 
\bar u - u \) is a descent direction.
Let us introduce the functional
\begin{align}
\mathcal{E}[u] &\doteq \left\langle \bar u- u, d[u] \right\rangle_{L^{2}}{\doteq \int_0^T \langle\bar u-u,d[u]\rangle\d t},\mbox{ where}\notag\\
d_j[u](t) & \doteq \int V^j_{t}(x, \mu_t[u])\cdot \d\nu_{t}[u](x), \quad t \in I, \quad 1\leq j \leq m.\label{d} 
\end{align}

It is clear that $\mathcal{E}[u] \geq  0$ and $\mathcal{E}[u] = 0$ implies that the pair $(\mu[u], u)$ satisfies the PMP. In other words, \( \mathcal{E} \) measures 
the 
``non-extremality'' of $u$.

Now, we can use the descent direction $\bar{u}-u$ for developing the following version of the classical backtracking algorithm.

\begin{algorithm}[H]
\caption{Descent method with backtracking ($k$-th iteration)}\label{alg:backtracking}
\begin{algorithmic}[1]
\Require \( u^k \in \mathcal{U} \), \( c,\theta\in (0,1) \).
\State Compute the trajectories \( \mu^k \) (forward in time) and \( \nu^{k} \) (backward in time starting from \eqref{eq:adj_term}) of the initial and the adjoint systems corresponding to \( u^{k} \), and take
 $d^{k}(t)\doteq (d_1, \ldots, d_m)[u^k](t),$
where $d_j[u]$ are introduced in \eqref{d}.

\State Compute the target control
\begin{align}\label{u-target}
  \bar u^{k}(t) \doteq  \arg\max_{\upsilon\in U}\left(\upsilon\cdot d^{k}(t)\right).
\end{align}

\State Set $\lambda^{k}\doteq \max\left\{\theta^{j}\mid j \geq 0,\; \mathcal{I}\left[u^{k}+\theta^{j}(\bar u^{k}-u^{k})\right]-\mathcal{I}[u^{k}]\leq c\, \theta^{j}\langle u^{k}-\bar u^{k}, d^{k} \rangle_{L^2}\right\}$.
\State Set \( u^{k+1}\doteq u^{k} + \lambda^{k}(\bar u^{k}-u^{k}) \).
\end{algorithmic}
\end{algorithm}

The convergence analysis of the algorithm is provided by the following theorem.

\begin{theorem}
  For any initial control $u^0\in \mathcal{U}$, the sequence $\{u^k\}$ generated by the algorithm
  \begin{enumerate}
      \item is monotone in the sense that
      \begin{align*}
      \mathcal{E}[u^k]\neq 0 \ \Rightarrow \ \mathcal{I}[u^{k+1}] < \mathcal{I}[u^{k}], \quad \mbox{ and } \quad
      \mathcal{E}[u^k]= 0 \ \Rightarrow \ \mathcal{I}[u^{k+1}] = \mathcal{I}[u^{k}]
      \end{align*}
      \item converges in the sense that
      \(
        \mathcal{E}[u^k]\to 0\) as \(k\to\infty.
      \)
  \end{enumerate}
\end{theorem}
\begin{proof}
Let \( e^{k} := \mathcal{E}[u^{k}] \) and assume that \( e^{k}\not\to 0 \) .
In this case, there exists \( \varepsilon>0 \) such that \( e^{k}\geq \varepsilon \) {for all indices $k$ from some countable set \( K\subset \mathbb N \)}.
By the choice of \( \lambda^k \), we obtain, for all \( k\in K \),
\begin{displaymath}
  \mathcal{I}[u^{k}] - \mathcal{I}[u^{k+1}] \geq c\lambda^{k} \left\langle \bar u^{k} - u^{k}, d^{k} \right\rangle_{L^{2}} = c\lambda^{k}e^{k}.
\end{displaymath}
This shows that \( (\lambda^{k})_{k\in K} \to 0 \), because otherwise \( \mathcal{I}[u^{k}] \to -\infty \) up to a subsequence.
For any large \( k \) we have \( \lambda^{k}\leq \theta \).
Hence \( \lambda := \lambda^{k}/\theta \) is an admissible step.
On the other hand,  by Step 4 of the algorithm for such \( \lambda \) we have:
\(
\mathcal{I}\left[u^{k}+\lambda(\bar u^{k}-u^{k})\right]-\mathcal{I}[u^{k}]\ge c\lambda\langle u^{k}-\bar u^{k}, d^{k} \rangle_{L^2},
\)
that is,
\begin{align*}
  -\frac{c\lambda^{k}}{\theta}\left\langle \bar u^{k} - u^{k}, d^{k} \right\rangle_{L^{2}}
  &\leq \mathcal{I}\left[u^{k}+(\lambda^{k}/\theta)(\bar u^{k} - u^{k})\right] - \mathcal{I}[u^{k}]\\
  &= - \frac{\lambda^{k}}{\theta}e^{k} + \mathcal{O}\left(\left(\lambda^{k}/\theta\right)^{2}\right),
\end{align*}
where we use the increment formula~\eqref{eq:incr2} to get the last equality.
Hence $ \frac{(1-c)\lambda^{k}}{\theta}e^{k}\leq C\left(\frac{\lambda^{k}}{\theta}\right)^{2},$
for some \( C>0 \), or equivalently,
\begin{displaymath}
  (1-c)\varepsilon\leq (1-c)e^{k}\leq C\frac{\lambda^{k}}{\theta}.
\end{displaymath}
Since the right-hand side tends to zero, we come to a contradiction.
\end{proof}

\subsection{Implementation}

In the algorithm described in Sect.~\ref{subsec:alg}, the primal and adjoint equations are solved numerically. If the original problem is periodic in space, and the driving vector field has a convolutional structure \eqref{VF-conv},
then, 
for the numeric integration, 
one gives preference to so-called spectral methods~\cite{Boyd2001}.

Assume that the initial measure \( \vartheta\in \mathcal{P}_{c}(\mathbb{R}) \) is absolutely continuous.
This implies that the corresponding trajectories \( \mu_t \) and \( \nu_t \) are absolutely continuous as well, and all ingredients of the algorithm can be recast in terms of their densities $\rho_t$ and $\zeta_t$, respectively. Moreover, since \( \vartheta \) is compactly supported, there exists a segment \( [a,b] \) such that
\(
\spt \mu_t \subset [a,b]\), and \(\spt \nu_t\subset [a,b]\) for all \(t\in I.
\)
This implies that 
\( \mu_{t} \) and \( \nu_t \) can be considered as measures on the circle \( \mathbb{S}^{1} \) 
(i.e. the measures can be view as $2\pi$-periodic in $x$). 

The primal and the adjoint equations can be written in the form:
  \begin{gather}
    \label{eq:simple}
    \partial_{t}\rho + \partial_{x}\left(V(x,\rho)\rho\right) = S(x,\rho),
  \end{gather}
  \[
  V(x,\rho) = f(x)+(K*\rho)(x),\quad
    S(x,\rho) = g(x)\rho(x)+h(x)(M*\rho)(x),\notag\\
  \]
  \( f,g,h,K,M\colon I \times \mathbb{S}^{1}\to \mathbb{R} \) are given functions. 

  Suppose that all the densities are of the class \( L^{2}(\mathbb{S}^{1}) \). Upon substitution of the truncated Fourier series 
  \begin{equation}
  \rho(x,t)=\sum_{n=-\mathcal{N}/2}^{\mathcal{N}/2} \widehat{\rho}_n(t) {\rm e}^{inx},\label{eq:f1}
  \end{equation}
  in (\ref{eq:simple}), the partial differential equation transforms into the system of ODEs
  \begin{equation}
  \frac{d\widehat{\rho}_{n}}{dt} = -in (\widehat{V\rho})_{n}+ \widehat{S}_{n},\quad n\in \mathbb{Z},\quad \| n \| \le \mathcal{N}/2,
  \label{eq:ode}    
  \end{equation}
  where the 
  hat over nonlinear terms 
  denotes their Fourier coefficients, and 
  \begin{equation}
  \widehat{\rho}_{n}(t) = \frac{1}{2\pi}\int_{0}^{2\pi}\rho(x,t)e^{-inx}\d x,
  \label{eq:f2}    
  \end{equation}
 stands for the Fourier coefficients of $\rho(x,t)$. 

 The system (\ref{eq:ode}) can be integrated by any appropriate numerical method (e.g. the Runge-Kutta method). Transformations between the physical and spectral (Fourier) spaces are computed by using the Fast Fourier Transforms (FFT). Multiplications of fields are usually computed in the physical space, derivatives and convolutions are evaluated in the Fourier space.

\subsection{Numerical experiment}

As an example, we consider the paradigmatic model of Kuramoto \cite{kuramoto2003chemical}, which describes an assembly of pairwise interacting homotypic oscillators.
Specifically, we consider an optimization problem in the spirit of \cite{Sinigaglia2021OptimalCO}, in which the goal is to synchronize a continuous oscillatory network by a given time moment \( T \).

The prototypic ODE representing the dynamics of $N$ oscillators takes the form
\begin{align}
    \dot x_i = \omega_i + u_1 +u_2 \, \frac{1}{N}\sum_{j=1}^N\sin(x_j-x_i-\alpha), \quad 1 \leq i \leq N. \label{Kuramoto-N}
\end{align}
Here, $x_i(t) \in \mathbb S^1$ and $\omega_i \in \mathbb{R}$ are the phase and natural frequency of the $i$th oscillator, respectively, $\alpha$ is the phase shift.
Control inputs are $t \mapsto u(t) \doteq (u_1(t), u_2(t))$, where $u_1$ affects the angular velocity, and $u_2$ modulates the connectivity of the network.

As in~\cite{Sinigaglia2021OptimalCO}, we assume that all oscillators have a common natural frequency $\omega$, which, in this case, can be specified as $\omega =0$.
As the number of oscillators $N \to \infty$, the limiting mean-field version of \eqref{Kuramoto-N} is described by the curve $t \mapsto \mu_t \in \mathcal{P}(\mathbb S^1)$ satisfying the nonlocal continuity equation driven by the vector field
\begin{align}
V(x, \mu, u) \doteq u_1+u_2\int_{0}^{2\pi} \sin(y -x -\alpha)\d \mu(y)\doteq  u_1 V^1 + u_2 V^2(x, \mu) .\label{Kuramoto-MF}
\end{align}
Consider the problem of steering the ensemble to a given phase $x_{0} + 2\pi n$, $n \in \mathbb Z$:
\begin{align*}
  \min\mathcal{I}[u] = \ell(\mu_T) \doteq & \,\displaystyle \int_{0}^{2\pi} J(x,x_{0}) \d \mu_T(x),\\
J(x,y)\doteq \frac{1}{2}(\sin x-\sin y)^2+&\,\frac{1}{2}(\cos x -\cos y)^2 = 1 - \cos(x-y).
\end{align*}

To specify the adjoint equation, we compute (see Example~\ref{ex:conv} in Sect.~\ref{subsec:deriv}):
\[
D_{x}V(x,\mu,u) = -u_2\int_{0}^{2\pi} \cos(x -y +\alpha)\d \mu(y), \quad D_{\mu}V(y,\mu,u,x) = u_2 \, \cos(y -x +\alpha),
\]
and \(  D_{\mu}\ell(\mu;x) = \sin(x-x_0) \).
Then, \eqref{eq:adjoint} becomes
\begin{align}
    \partial_t \nu + \partial_x\left(v \, \nu\right)=u_2\left[(K_1*\nu)\mu + (K_2*\mu)\nu \right],\label{Kuramoto-1-dual}
\end{align}
where \( v_{t}(x) = V_{t}(x,\mu_{t},u(t)) \), \( K_1(x) = \cos(-x+\alpha) \), \( K_2(x) = \cos(x+\alpha) \). 

Let us associate \( \mu \) and \( \nu \) with their densities represented as in (\ref{eq:f1}) in terms of the Fourier coefficients \( \widehat{a}_n \) and \( \widehat{b}_n \), respectively. To represent the PDE~\eqref{eq:conteq} in the Fourier space (i.e. in the form (\ref{eq:ode})), notice that the only non-vanishing Fourier coefficients of (\ref{Kuramoto-MF}) are 
$\widehat{V}_0=u_1,\ 
\widehat{V}_1=i \pi u_2\widehat{\mu}_1\exp(i\alpha)$,
and the complex conjugate of the latter one is $\widehat{V}_{-1}$.
This form of $V(x,\mu,u)$ enables us to compute the r.h.s. of (\ref{eq:ode}) exclusively in the Fourier space with no recourse to the physical space, in contrast with the case when applying the pseudospectral methods to the system with a generic $V(x,\mu,u)$.     

In the Fourier space, the nonlocal continuity equation reads
\begin{equation}
\label{eq:eqa}
\frac{d\widehat{a}_n}{dt} =
-inu_{1}\widehat{a}_n + \pi n u_{2} \left( 
\widehat{a}_{1}\widehat{a}_{n-1}{\rm e}^{i\alpha} -
\widehat{a}_{-1}\widehat{a}_{n+1}{\rm e}^{-i\alpha}
\right),\quad n\in \mathbb{Z},
\end{equation}
while the adjoint equation~\eqref{eq:adjoint} and the terminal condition~\eqref{eq:adj_term} become:
{
\begin{align}
\label{eq:eqb}
\frac{d\widehat{b}_{n}}{dt} &= -inu_{1}\widehat{b}_n +\pi n u_2 \left(
\widehat{a}_1\widehat{b}_{n-1}{\rm e}^{i\alpha}-
\widehat{a}_{-1}\widehat{b}_{n+1}{\rm e}^{-i\alpha}
\right)\nonumber\\
&+ \pi u_2\left(
(\widehat{a}_1\widehat{b}_{n-1}+\widehat{b}_{-1}\widehat{a}_{n+1}){\rm e}^{i\alpha}+
(\widehat{a}_{-1}\widehat{b}_{n+1}+\widehat{b}_{1}\widehat{a}_{n-1}){\rm e}^{-i\alpha}
\right),\\ 
\widehat{b}_n(T)& \ = \frac{i}{2}\left(
\widehat{a}_{n-1}(T){\rm e}^{-ix_{0}} -
\widehat{a}_{n+1}(T){\rm e}^{ix_0}
\right),\quad n\in \mathbb{Z}.\notag
\end{align}
}
In order to compute the transformation $\rho(x_j,t)\mapsto\widehat{\rho}_k(t)$ and its inverse, we employ the forward and backward FFTs implemented in the library FFTW \cite{FFTW}.

The problem is considered under the control constraint \( u_1^2+u_2^2\leq 2 \); for the $k$th iteration, the corresponding target control $(\bar u_1^k, \bar u_2^k)$ provided by \eqref{u-target} takes the form: $\frac{\sqrt{2}}{\sqrt{(d_1^k)^2 + (d_2^k)^2}}(d_1^k, d_2^k)$, and the control-update rule reads: $u^{k+1} = u^k+\lambda^k(\bar u^k - u^k)$.


Some computational results 
are presented by Fig.~\ref{fig:kuramoto}.

\begin{figure}[h]
  \begin{center}
\includegraphics[width=0.44\textwidth]{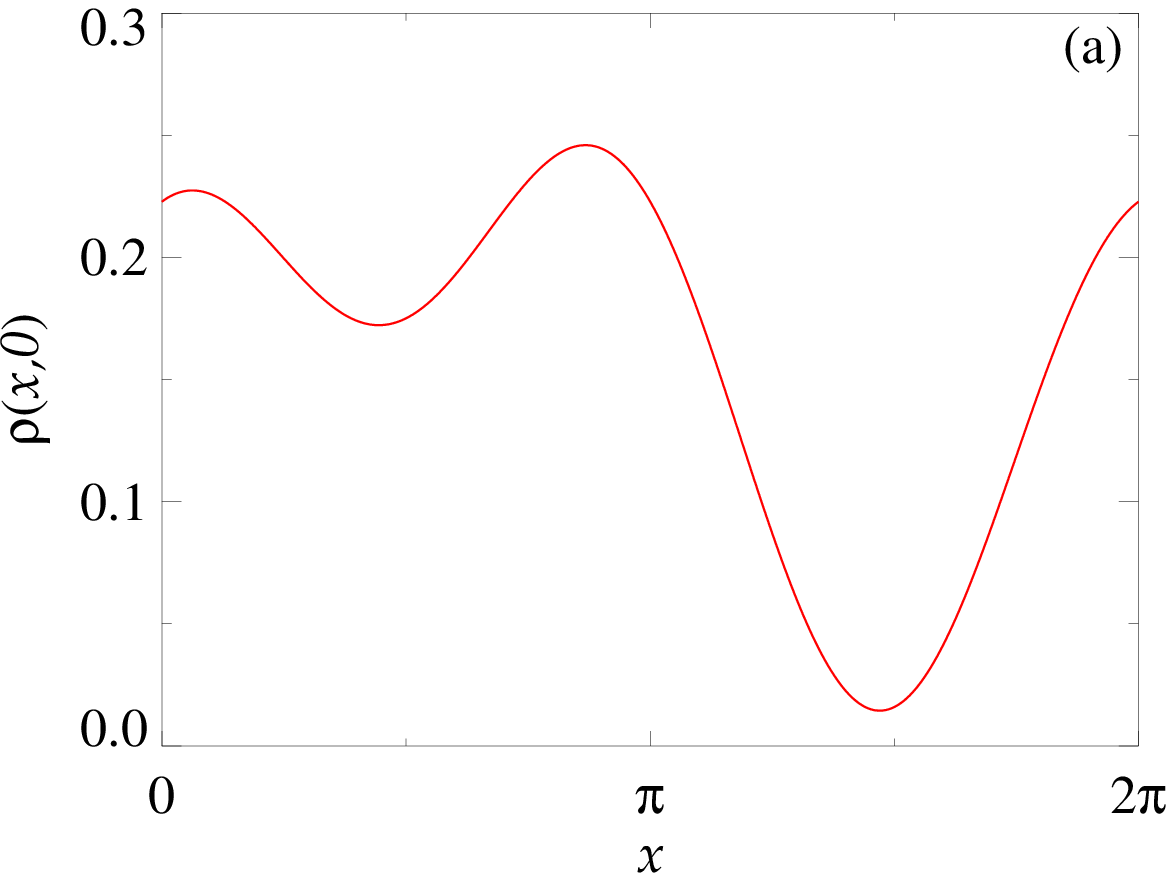}  \includegraphics[width=0.44\textwidth]{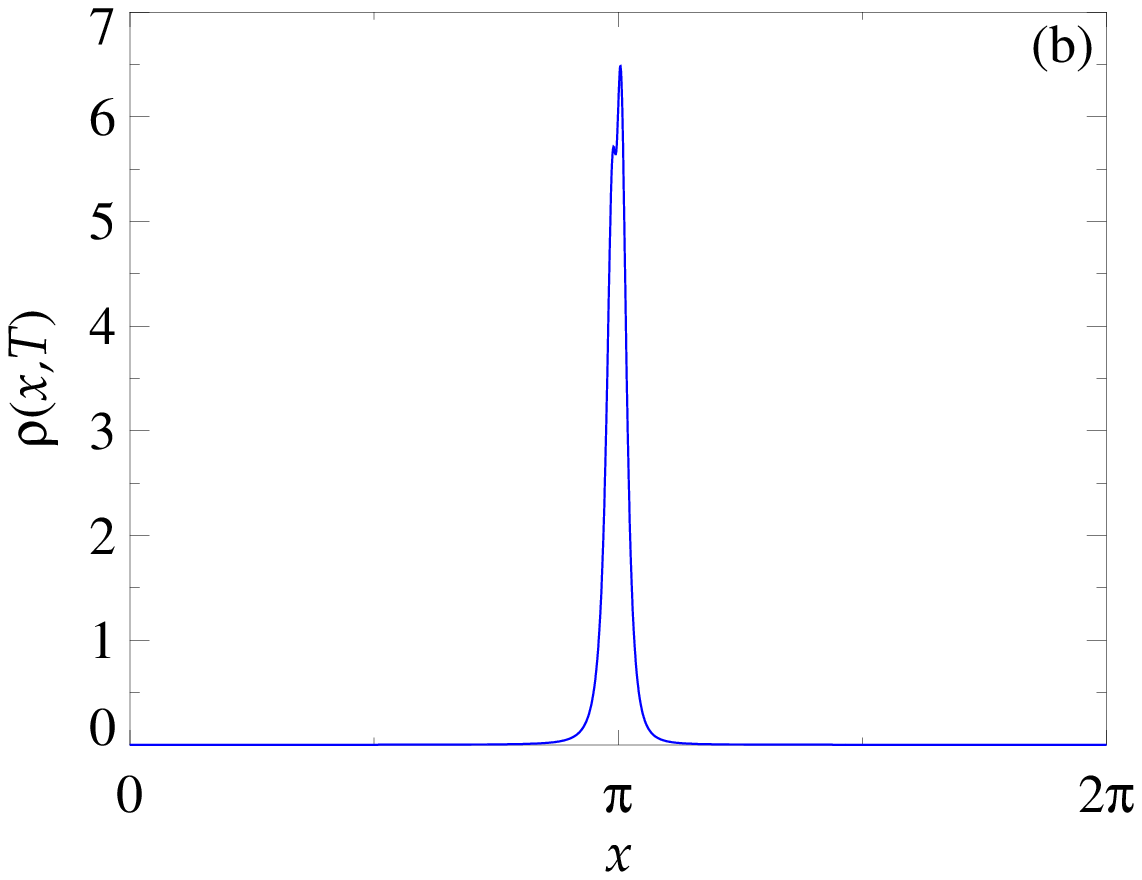}
\includegraphics[width=0.44\textwidth]{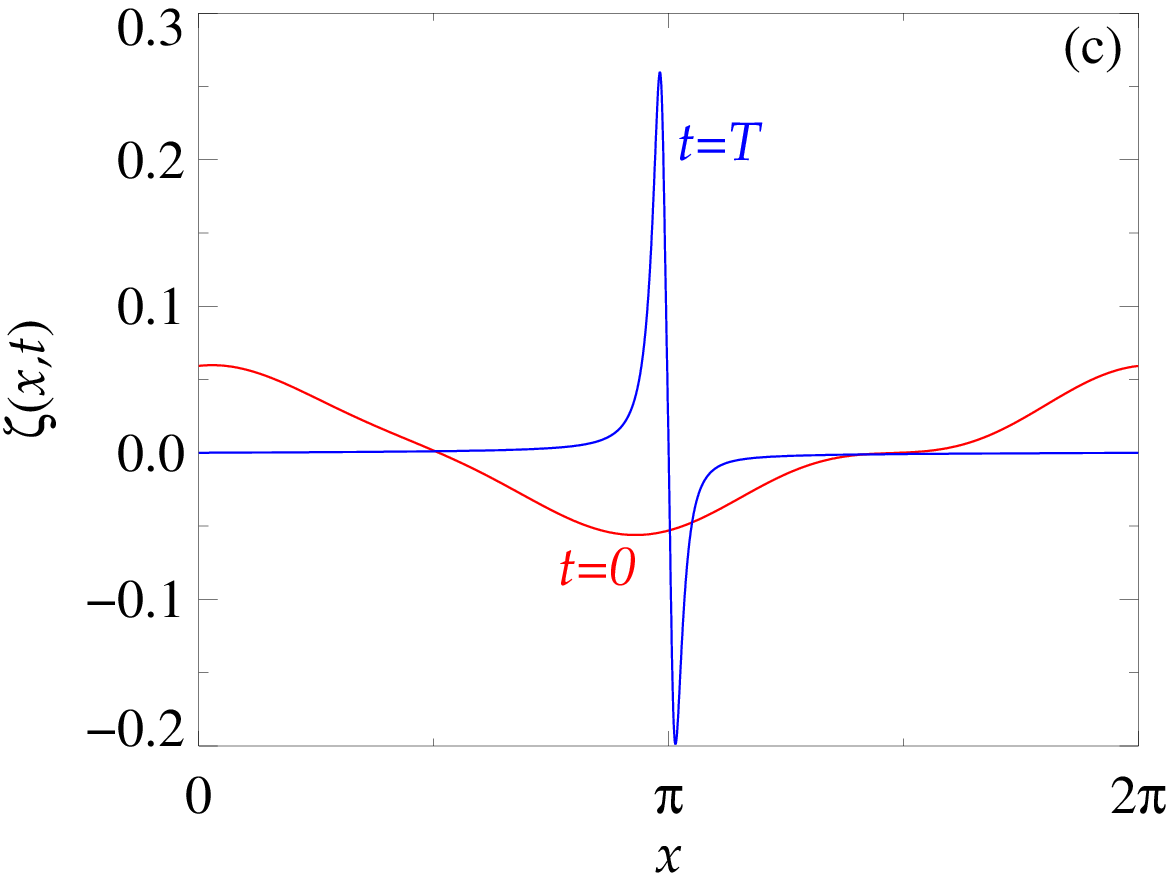}
\includegraphics[width=0.44\textwidth]{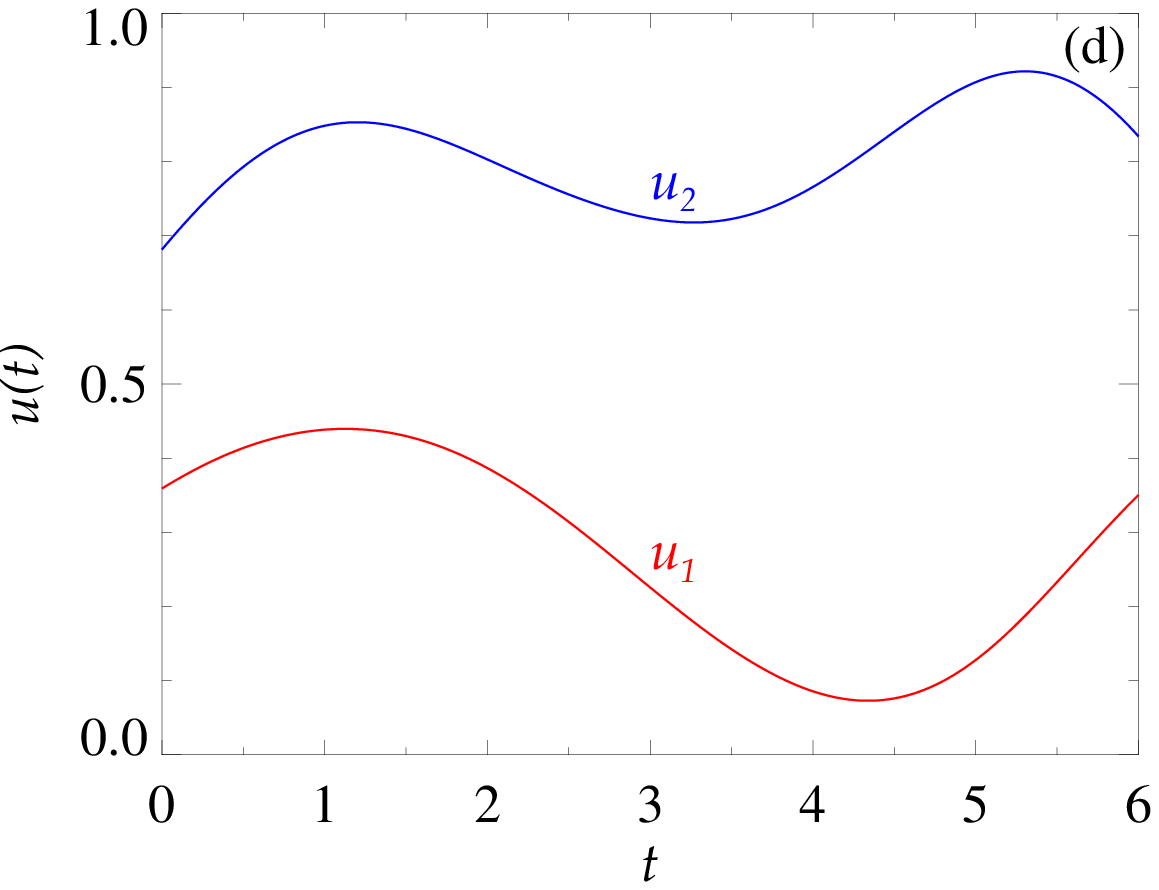}
  \end{center}
  \caption{\small Numerical solution of the synchronization problem for $T=6$,
  $\alpha=0$, $x_0 = \pi$, the initial control $u_1^0=\sqrt{2}\sin(2\pi t),\ u_2^0 =
  \sqrt{2}\cos(2\pi t)$, and the initial distribution $\rho_{0}(x) = \left( 2+\sin x +
  0.8\cos 2x -0.2 \sin 2 x  \right)/(4\pi)$. The trajectory $\rho^k_t$ produced by the algorithm is depicted in (a) at $t=0$ and in (b) at $t=T$.
    The adjoint trajectory $\zeta^k_t$ at the same time moments is presented in (c).
    The corresponding controls $u^k_1(t)$ and $u^k_2(t)$ are shown in (d).
    For these controls, the cost functional is $\mathcal{I}[u^k]\approx 8\cdot10^{-3}$ (cf. $\mathcal{I}[u^0]=1$). 
    In computations, we used $\mathcal{N}=2048$ spatial Fourier harmonics, and $c = 0.01$, 
    $\theta = 0.5$ for determination of $\lambda^k$. 
    The systems of ODEs (\ref{eq:eqa}) and (\ref{eq:eqb}) were solved by the 4th-order Runge-Kutta method with constant time step $\tau=0.001$; 
    stopping criterion: $\lambda^{k}<10^{-2}$. 
    }
  \label{fig:kuramoto}
\end{figure}
  \begin{remark}
    Let us stress several differences between the problem that we solve here and the one addressed in \cite{Sinigaglia2021OptimalCO}.
    First of all, in \cite{Sinigaglia2021OptimalCO} the authors consider the so-called mean-field type controls, i.e., they assume that \( u \) depends not only on \( t \) but also on \( x \). It is clear that this choice greatly improves the controllability of the system.
    Moreover, the system in \cite{Sinigaglia2021OptimalCO} is subject to common noise, which also contributes to the controllability.
    Indeed, let the initial density be given by \( \rho_{0} = 1 + \sin kx \), \( k \ge 2 \).
    Then, the convolution in~\eqref{Kuramoto-MF} vanishes, which means that our control options reduce to shifting the wave \( \rho_0 \) back and forth. On the other hand, under the presence of common noise, the Fourier coefficient corresponding to \( \sin x \) immediately becomes  nonzero and, as a result, the system is self-synchronizing for any positive \( u_2 \).
    A similar effect can be observed if we try to solve~\eqref{eq:conteq}, \eqref{Kuramoto-MF} with a discretization scheme that involves a numerical diffusion (such as the classical Lax-Friedrichs method).
  \end{remark}

  \section{General case}
\label{sec:gen}

In this section, we shall discuss a natural extension of the obtained results to the control-nonlinear case and general cost functional \eqref{eq:cost}.

\subsection{Nonlinear dependence on control}
\label{subsec:gen_nonlin}

To handle the case of nonlinear dependence $u \mapsto V_t(x, \mu, u)$, we shall resort to the standard technique based on the extension of the original class $\mathcal U$ of control signals to a broader space
$
\widetilde{\mathcal U} \doteq
\left\{\eta\in \mathcal P(I\times U): \, \left[(t, u) \mapsto t\right]_{\sharp}\eta = \frac{1}{T}\mathcal L^1\right\}
$
of Young measures \cite{Valadier}.
It is well-known that such an extension provides the linearization of the vector field w.r.t. the driving signal and, in a certain sense, reduces the general model to the above control-affine case.
Recall that \textit{i)} $\mathcal U$ is dense in $\widetilde{\mathcal U}$ due to the embedding $u \mapsto \eta$, $\eta_t=\delta_{u(t)}$, where $t \mapsto \eta_t \in \mathcal P(U)$ is the weakly measurable family of probability measures obtained by disintegration of $\eta$ w.r.t. $\frac{1}{T}\mathcal L^1$;  and \textit{ii)} $\widetilde{\mathcal U}$ is compact in the topology of weak convergence of probability measures (and therefore, in any metric $W_p$, $p\geq 1$)  as soon as $U$ is compact, thanks to the classical Prohorov theorem.

This passage, which is a routine of the mathematical control theory,
leads to the following relaxation of the original dynamics \eqref{eq:conteq}:
\begin{gather}
\partial_t\mu_t + \div \left(\widetilde V_{t}\left(\cdot,\mu_{t},\eta_t\right) \, \mu_t\right) = 0,\quad \mu_0=\vartheta,\label{eq:conteq-relaxed}\\
\widetilde V_{t}\left(x,\mu_{t},\eta_t\right) \doteq \int_{U} V_{t}\left(x,\mu_{t},u\right)\d \eta_t(u)\doteq \langle \eta_t, V_{t}\left(x,\mu_{t},\cdot\right)\rangle;\notag
\end{gather}
the original cost should be reformulated in the corresponding form:
$
\widetilde{\mathcal I} = \ell(\tilde\mu_T),
$
where $\tilde\mu_t$ is a solution of \eqref{eq:conteq-relaxed}.

Observing that the dependence $\omega \mapsto \widetilde V_{t}\left(x,\mu,\omega\right)$ is linear, we 
invite the reader to consider the weak variation
\(
\eta^\lambda = \eta + \lambda(\bar\eta - \eta)
\)
and the respective cost increment $\widetilde{\mathcal I}[\eta^\lambda] - \widetilde{\mathcal I}[\eta]$ in place of \eqref{eq:u-lambda} and \eqref{eq:increment_draft}, and reproduce the arguments of Sect.~\ref{sec:incr} and \ref{sec:adjoint}.
By doing this, one ensures that the resulting increment formula and necessary condition for the optimality of a Young measure $\eta$ keep the form of Theorems~\ref{thm:increment} and \ref{thm:pmp2}, where $V$ and $\mathbf H_{t}$ are replaced by $\widetilde V$ and $\widetilde{\mathbf H}_{t}$, respectively,
\(
\displaystyle\widetilde{\mathbf H}_{t}(\mu, \nu, \omega) \doteq  \int_{U}\mathbf H_{t}(\mu, \nu, u) \d\omega(u),
\)
and the maximum condition \eqref{eq:maximum-cond-limit2} becomes
\begin{align*}
    \widetilde{\mathbf H}_{t}(\tilde\mu_t, \tilde\nu_t, \eta_t) = \max_{\omega \in \mathcal P(U)}\widetilde{\mathbf H}_{t}(\tilde\mu_t, \tilde \nu_t, \omega) \quad &\Leftrightarrow \quad \spt(\eta_t) \subseteq \arg\max_{\upsilon \in U} \mathbf H_{t}(\tilde\mu_t,\tilde\nu_{t}, \upsilon),
\end{align*}
where $\tilde\nu_t$ is the adjoint backward solution associated to $\eta$.
Now, if the addressed control-nonlinear problem $(P)$ does have a usual minimizer $u\in\mathcal{U}$, then PMP for $u$ is restored by taking $\eta$ such that $\eta_t =\delta_{u(t)}$.

\subsection{Running cost}
\label{subsec:gen_run}

If the map $u \mapsto L_t(x, \mu, u)$ is affine, one easily adapts PMP by reformulating the dynamics \eqref{eq:DH} of the Hamiltonian PDE and the Hamiltonian \eqref{eq:hamilt2} as
\begin{equation}
\label{eq:DH-L}
  \vec{H}\doteq
  \begin{pmatrix}
    \displaystyle V_{t}\\[0.2cm]
    \displaystyle- p\, D_x \, V_t - \iint  q \, D_\mu V_t\d \gamma + D_x \, L_{t} + \iint  D_\mu L_t\d \gamma
  \end{pmatrix},
\end{equation}
and $     \displaystyle  \mathbf H_{t} \doteq  \int V_{t}\cdot \d\nu-L_{t}.$
Further details can be found, e.g., in~\cite{BonnetFrankowska2021a}.
The general $u$-nonlinear case refers to the relaxation technique exhibited in Sect.~\ref{subsec:gen_nonlin}.

\section*{Declarations}

\bmhead{Acknowledgments}
We are grateful to the anonymous referees for their valuable comments enabling us to significantly improve the paper.

\bmhead{Conflict of interest}
The authors have not disclosed any competing interests.

\bmhead{Funding} 
RC and MS acknowledge the financial support of the Foundation for Science and Technology (FCT/MCTES) in the framework of the Associated Laboratory -- Advanced Production and Intelligent Systems (AL ARISE, ref. LA/P/0112/2020), the R\&D Unit SYSTEC (Base UIDB/00147/2020 and Programmatic UIDP/00147/2020 funds), and projects RELIABLE -- Advances in control design methodologies for safety critical systems applied to robotics (ref. PTDC/EEI-AUT/3522/2020) and MLDLCOV -- Impact of confinement measures related to COVID-19 on mobility, air pollution and macroeconomic indicators in Portugal: an approach in Machine Learning (ref. DSAIPA/CS/0086/2020), the latter through the program INCO.2030 -- National Initiative for Digital Competences e.2030.
A part of the simulations was carried out with the OBLIVION Supercomputer (at the High Performance Computing Center, University of Évora) funded by the ENGAGE SKA Research Infrastructure (reference POCI-01-0145-FEDER-022217 - COMPETE 2020 and the FCT, Portugal) in the framework of the FCT calls for computational projects (refs. 2021.09815.CPCA and 2022.15706.CPCA.A2).

\bibliographystyle{abbrv}
\bibliography{sn-bibliography}

\end{document}